\documentclass[12pt]{article}
\usepackage{amsmath, graphicx, amsfonts,amssymb, calrsfs}
\usepackage{amsfonts,mathrsfs, color, amsthm}
\usepackage{indentfirst}

\usepackage{accents}

\usepackage{enumitem}

\newtheorem{theorem}{Theorem}[section]
 \newtheorem{corollary}[theorem]{Corollary}
 \newtheorem{lemma}[theorem]{Lemma}
 \newtheorem{proposition}[theorem]{Proposition}
 \theoremstyle{definition}
 \newtheorem{definition}[theorem]{Definition}
 \theoremstyle{remark}
 \newtheorem{remark}[theorem]{Remark}
 
 \numberwithin{equation}{section}

\def\bpf{\begin{proof}}
\def\epf{\end{proof}}
\def\be{\begin{equation}}
\def\ee{\end{equation}}
\def\bea{\begin{eqnarray}}
\def\eea{\end{eqnarray}}
\def\bt{\begin{theorem}}
\def\et{\end{theorem}}
\def\bl{\begin{lemma}}
\def\el{\end{lemma}}
\def\br{\begin{remark}}
\def\er{\end{remark}}
\def\bc{\begin{corollary}}
\def\ec{\end{corollary}}
\def\bd{\begin{definition}}
\def\ed{\end{definition}}
\def\bp{\begin{proposition}}
\def\ep{\end{proposition}}

\begin{document}
\title{Extremal function for a sharp Moser-Trudinger type inequality on the upper half space
}

\author{Yubo Ni}
\date{}
\maketitle
\begin{abstract}
We establish a sharp Moser-Trudinger type inequality on the upper half space for two dimensions.~Then we investigate the existence of the extremal functions for this sharp Moser-Trudinger type inequality under dynamic changes in the unit ball. Additionally, we derive a logarithmic inequality named weighted Euclidean Onofri inequality and another Trudinger type inequality.
 \end{abstract}

\section{Introduction}
In this paper we establish a sharp Moser-Trudinger type inequality in~${\mathbb{R}}^{2}$~with partial monomial weight~$t^{\alpha}$,~$t^{\beta}$.~Then we study the existence of the extremal functions for this sharp Moser-Trudinger type inequality under dynamic changes in the unit ball.~For inequality with monomial weight,~there was a significant study described below.~The open question raised by Brezis~\cite{BI2003,BB2003}~about the problem of the regularity of stable solutions to reaction--diffusion problems of double revolution was considered.~This problem was studied by Cabr\'{e} and Ros-Oton~\cite{CR2013}~and then they established the Sobolev,~Morrey,~Trudinger~and isoperimetric inequalities with monomial weight~\cite{CS2013}.~Bakry,~Gentil~and Ledoux~\cite{B2013}~proved the Sobolev inequality with monomial weight,~and moreover,~the best constant~$S(n,a)$~was also calculated.~Nguyen~\cite{N2015}~proved the same result again using the mass transport approach,~thereby extending the above result.~Furthermore,~he also studied the best constants and extremal functions for the Gagliardo-Nirenberg inequalities and logarithmic Sobolev inequalities with the weight~$X_{n}^{a}$~for an arbitrary norm.~Cabr\'{e}~and~Ros-Oton~considered the borderline cases of Sobolev inequalities;~therefore,~they established the monomial Trudinger inequality~\cite{CS2013}.~Lam was inspired by Brezis~\cite{BI2003,BB2003},~Cabr\'{e}~and Ros-Oton~\cite{CR2013}.~Then the sharp Moser-Trudinger inequality with monomial weight~$X^{A}$~was established.~However,~are there some extremal functions for this sharp Moser-Trudinger inequality with monomial weight?~If we figure out the existence problem of the optimal constant and extremal function,~we can then study the corresponding problem of prescribing the Gaussian curvature~\cite{H1990}~and prove the existence of solutions to mean-field equations~\cite{C2010}.~This is of great significance for the development of geometric analysis~\cite{Y2016}\cite{Z2019}.~This paper devise ideas with the aim of studying this problem.

The main difficulty is that the lack of compactness of the Moser-Trudinger embedding prevented the maximizing sequences from converging.~Through symmetrization arguments,~the functions could be transformed into radial functions~and sequences lacking compactness were concentrated at the origin.~For the domain as a ball in~$\mathbb{R}^{n}$,~Carleson and Chang estimated the supremum of the functionals of all maximal sequences and constructed a test function whose functional exceeded this supremum.~Therefore,~it has been shown that a convergent maximum sequence existed.~Subsequently,~Struwe~\cite{S1988}~established the results for bounded domains that were close to the balls.~Flucher~\cite{F1992}~proved the result for a bounded domain in~$\mathbb{R}^{2}$~and the results were further extended to higher dimensions~\cite{L1996}.~Adimurthi~\cite{A2004},~Struwe~\cite{A2000}~and Li~\cite{L2001}~developed a blow-up analysis method to solve this problem.~In recent years,~more researches on the existence of minimizers of Moser-Trudinger inequalities have been reported in~\cite{I2019},~\cite{L2022}~and~\cite{N2022}.~Bubbling nodal solutions of the Moser-Trudinger equation have also been studied~\cite{G2021}.~For our purpose,~we will establishes concentration-compactness principle which can be traced back to Lions~\cite{L1985}~who proved the concentration-compactness principle for the classical Moser-Trudinger inequality.~Then we will prove the existence of minimizers of a weighted Moser-Trudinger inequality in the two-dimensional upper half space under dynamic changes.~For the existence of minimizers of a weighted Moser-Trudinger inequality,~we studied the concepts from Carleson and Chang~\cite{C1986}~and the techniques of Roy~\cite{R2016}~for the existence of minimizers of the logarithmic weighted Moser-Trudinger inequality.~

To achieve our purpose,~we first construct a sharp weighted Moser-Trudinger inequality.~The Trudinger inequality can be traced back to Trudinger's work in 1967~\cite{T1967}.~Let~$\Omega$~be a bounded domain in space~$\mathbb{\mathbb{R}}^{n}$.~The classical Sobolev continuous embedding be asserted as
~$W_{0}^{1,p}\hookrightarrow L^{q}(\Omega)$~\cite{G1977}~for a positive integer~$n$,~$1\leq p\leq p^{*}=\frac{np}{n-p}$~and~$1\leq q\leq p^{*}$.~The Moser-Trudinger inequalities concern borderline cases
\begin{align*}
p=n,~for~which~formally~1\leq q\leq p^{*}=\infty,
\end{align*}
which leads to the question as to whether
\begin{align*}
W_{0}^{1,p}(\Omega)\subset L^{\infty}(\Omega).
\end{align*}
We can conclude the answer is no and simple examples can be found in~\cite{E2010}.~However,~Trudinger proved that~$W^{1,n}(\Omega)\subseteq L_{\varphi_{n}}(\Omega)$~is continuously embedded into the Orlicz space~$L_{\phi^{*}}(\Omega)$~with growth function~$\phi(t)=e^{{|t|}^{\frac{N}{N-1}}}-1$,~where~$\Omega\subseteq R^{n}$~is a bounded domain and~$L_{\phi}(\Omega)$~is defined as the set of all functions~$u$~defined on~$\Omega$~satisfying~$\displaystyle{\int_{\Omega}\phi(u(x))dx<\infty}$.~The Orlicz space~$L_{\phi^{*}}(\Omega)$~is the linear hull of ~$L_{\phi}(\Omega)$~in the Luxembourg norm
\begin{align*}
\displaystyle{{\|u\|}_{L_{\phi^{*}}(\Omega)}=\inf\left\{\lambda>0:\int_{\Omega}\phi(\frac{u(x)}{\lambda})dx\leq1\right\}}.
\end{align*}
Trudinger~\cite{T1967}~proved the inequality through an expansion of the power series of the exponential function.~Let~${\|\nabla u\|}_{n}=\displaystyle{\left(\int_{\Omega}{|\nabla u|}^{n}dx\right)}^{\frac{1}{n}}$,~then there exists two constants~$\alpha>0$~and~$c_{0}>0$~such that
\begin{equation*}
\frac{1}{|\Omega|}\displaystyle{\int_{\Omega}e^{\left(\alpha|u|^{\frac{n}{n-1}}\right)}dx\leq c_{0}},
\end{equation*}
for any~$u\in W_{0}^{1,n}(\Omega)$~which satisfies~$\|\nabla u\|_{n}\leq 1$.~Subsequently,~Moser~\cite{M1971}~proved the sharp form of the Trudinger inequality in a bounded domain.~Let~$\omega_{n-1}$~be the~$(n-1)$~dimensional surface of the unit sphere.~Then there exists a constant~$c_{0}$~that depends only on~$n$~such that for any~$u\in C^{1}_{c}(\Omega)$~and~$\alpha\leq\alpha_{n}=n{\omega_{n-1}}^{\frac{1}{n-1}}$,~
\begin{equation*}
\sup\limits_{\int_{\Omega}|\nabla u|^{n}dx\leq1}\displaystyle{\int_{\Omega}e^{\alpha|u|^{\frac{n}{n-1}}}}dx\leq c_{0}|\Omega|.
\end{equation*}
The integral on the left actually is finite for any positive~$\alpha$.~But if~$\alpha>\alpha_{n}$,~it can be made arbitrarily large by an appropriate choice of~$u$.~Moser further proved the following inequality in a two-dimensional sphere.~If~$u$~is a smooth function defined on~$S^{2}$~satisfying~$\displaystyle{\int_{\rm{s}^{2}}{|\nabla u|}^{2}}\\$${d\mu\leq1}$~and~$\displaystyle{\int_{\rm{s}^{2}}ud\mu=0}$~where~d$\mu$~denotes sphere area element,~then there exists a constant~$c$~such that
\begin{align*}
\int_{\rm{s}^{2}}e^{4\pi u^{2}}d\mu\leq c.
\end{align*}
We observed some works about Divergent operator with degeneracy and related sharp inequalities.~Dou,~Sun,~Wang and Zhu~\cite{D2022}~studied the following equation:
\begin{align*}
-\rm{div}(t^{\alpha}\nabla u)=t^{\beta}{|u|}^{p-1}.
\end{align*}
They provided direct proof of inequality as
\begin{align*}
\displaystyle{{\left(\int_{\mathbb{R}_{+}^{n+1}}t^{l}{|u|}^{\frac{n+l+1}{n+k}}dydt\right)}^{\frac{n+k}{n+l+1}}\leq C_{n,k}\int_{\mathbb{R}_{+}^{n+1}}t^{k}|\nabla u|dydt},
\end{align*}
for~$u\in C_{0}^{\infty}(\overline{\mathbb{\mathbb{R}}_{+}^{n+1}})$,~$\mathbb{\mathbb{R}}_{+}^{n+1}=\{(y,t)\in\mathbb{ \mathbb{R}}^{n+1}:t>0\}$,~$n\geq1$,~$l>-1$,~$k>0$~and~$\frac{nl}{n+1}\leq k\leq l+1$.~This is incontrovertible even for~$k\leq0$~which can be seen in~Maz'ya~\cite{M2011}.~The constant can be obtained for the inequality
\begin{align}\label{1.1}
\displaystyle{{\left(\int_{\mathbb{R}_{+}^{n+1}}t^{\beta}{|u|}^{p^{*}}dydt\right)}^{\frac{2}{p^{*}}}\leq c_{1,\alpha,\beta}\int_{\mathbb{R}_{+}^{n+1}}t^{\alpha}{|\nabla u|}^{2}dydt},
\end{align}
where~$u\in D_{\alpha}^{1,2}(\mathbb{R}_{+}^{n+1})$~is the completion of the space~$C_{0}^{\infty}(\overline{\mathbb{R}_{+}^{n+1}})$~under the norm~$\|u\|_{D_{\alpha}^{1,p}}(R_{+}^{n+1})$$={\left(\displaystyle{\int_{R_{+}^{n+1}}t^{\alpha}{|\nabla u|}^{p}}dydt\right)}^{\frac{1}{p}}$.~Dou,~Sun,~Wang~and Zhu~\cite{D2022}~proved that a constant could be reached by a nonnegative extremal function.~

Inspired by the weighted Sobolev inequality~\eqref{1.1},~we considers the case of the exponential end of the weighted Sobolev inequality,~that is~$p_{*}=\frac{2(n+\beta+1)}{n+\alpha-1}$,~$n+\alpha\rightarrow1$.~Therefore we obtain a new weighted Moser-Trudinger inequality as follows:

Here,~$c_{\alpha}=\displaystyle{\int_{0}^{\pi}(\sin\theta)^{\alpha}d\theta}$,~$c_{\beta}=\displaystyle{\int_{0}^{\pi}(\sin\theta)^{\beta}d\theta}$~and~$m_{\beta}(E)=\displaystyle{\int_{E}t^{\beta}dx}$.~
\begin{theorem}\label{theorem1.1}
Let~$\alpha,\beta>-1$~and~let~$\Omega\subset \mathbb{\mathbb{\mathbb{R}}}_+^2$~be a bounded domain.~Then there exists a constant~$c_{0}=c_{0}(\alpha,\beta)$~such that for each~$u\in C_c^\infty(\Omega)$~with~$\displaystyle{\int_{\Omega}|\nabla u|^{2+\alpha}t^{\alpha}}$dxdt~$\leq1$,
\begin{align}
\displaystyle{\int_{\Omega}\left(e^{a_{\alpha,\beta}{|u|}^{\frac{2+\alpha}{1+\alpha}}}-1\right)t^{\beta}dxdt\leq (2+\beta)c_{0}m_{\beta}(supp(u))},\label{1.2}
\end{align}
where\begin{equation*}
a_{\alpha,\beta}=(2+\beta)\left(2\pi^{\frac{1}{2}}\frac{\Gamma(\frac{\alpha+1}{2})}{\Gamma(\frac{\alpha}{2}+1)}\right)^{\frac{1}{1+\alpha}}
\end{equation*}
is the best constant.
\end{theorem}
Next the result about existence of the extremal functions for a sharp Moser-Trudinger type inequality under dynamic change in this paper as follow.

We write~$B^{+}$~by the two-dimensional unit upper hemisphere which centre at the origin~and write~$D_{0,\alpha}^{1,2+\alpha}(B^{+})$~by the weighted Sobolev space as the completion of spaces~$C_{0}^{\infty}(B^{+})$~under the norm
\begin{align*}
\|u\|_{D_{0,\alpha}^{1,2+\alpha}(B^{+})}=\displaystyle{\left(\int_{B^{+}}|\nabla u|^{2+\alpha}t^{\alpha}dxdt\right)^{\frac{1}{2+\alpha}}},
\end{align*}
where~$D_{0,\alpha,rad}^{1,2+\alpha}(B^{+})$~denotes a subspace composed of the radial functions in~$D_{0,\alpha}^{1,2+\alpha}(B^{+})$.~

Further,~we write~$J(u)=\frac{1}{(2+\beta)c_{0}m_{\beta}(B^{+})}\displaystyle{\int_{B^{+}}\left(e^{a_{\alpha,\beta}{|u|}^{\frac{2+\alpha}{1+\alpha}}}-1\right)t^{\beta}dxdt}$,~~$J^{\delta}_{\alpha,\beta}(0)=\sup\limits_{\{w_{m}\}_{m}}$
$\{\lim\limits_{m\rightarrow\infty}\sup J(w_{m})||\nabla w_{m}|^{2+\alpha}$.~$t^{\alpha}\rightharpoonup\delta_{0}\}$,~$c_{\varsigma,\alpha}={(1-{(1+\varsigma)}^{-1-\alpha})}^{-\frac{1}{1+\alpha}}$~where~$\varsigma\in \mathbb{R^{+}}$.
\begin{theorem}\label{theorem1.2}
Assume that~${(\frac{1+\alpha}{2+\alpha})}^{\frac{1+\alpha}{2+\alpha}}2^{\frac{1}{2+\alpha}}\leq \frac{1}{c_{\varsigma,\alpha}}$,~$\frac{1+\varsigma}{c_{\varsigma,\alpha}}<1$.~For~$\beta>-1$,~there exists a function~$u_{0}\in D_{0,\alpha,rad}^{1,2+\alpha}(B^{+})$~such that~
\begin{equation*}
J(u_{0})=\sup\limits_{{\|u\|}_{D_{0,\alpha}^{1,2+\alpha}(B^{+})}\leq \frac{1}{c_{\varsigma,\alpha}}}J(u).
\end{equation*}
\end{theorem}
The remainder of this paper is organized as follows.~Section 2 presents a weighted sharp Moser-Trudinger inequality,~a logarithmic inequality named weighted Euclidean Onofri inequality and another weighted Moser-Trudinger inequality on bounded domains.~In Section 3,~the principle of concentration compactness is established.~Finally,~the proof of the main assertion regarding the existence of an extremal function of the weighted Moser-Trudinger inequality under dynamic change in the unit ball is detailed in Section 4.
\section{Weighted sharp~Moser-Trudinger~inequalities}
The following Lemmas are required to prove the main theorems.
\begin{lemma}[Adams 1988 \cite{A1988}]\label{lemma2.1}
Let~$1<p<\infty$,~$\frac{1}{p^{'}}+\frac{1}{p}=1$~and let~$a(r,s)$~be a non-negative measurable function on~$[0,\infty)\times[0,\infty)$~such that almost everywhere
$$~a(r,s)\leq1,\makebox{when}~0<r<s~,$$
\begin{equation*}
\sup\limits_{\rm{s}>0}\left(\int_s^\infty a(r,s)^{p'}dr\right)^{\frac{1}{p'}}=b<\infty.
\end{equation*}
If~$\phi\geq0$~and~$\displaystyle{\int_{0}^{\infty}\phi(r)^{p}dr\leq1}$,~where\begin{equation*}F(s)=\displaystyle{s-\left(\int_{0}^{\infty}a(r,s)\phi(r)dr\right)^{p'}}.\end{equation*}
Then there exists a constant~$c_{0}=c_{0}(p,b)$~such that
\begin{equation*}\int_{0}^{\infty}e^{-F(s)}ds\leq c_{0}.\end{equation*}
\end{lemma}
\begin{lemma}[Lam 2017 \cite{L2017}]\label{lemma2.2}Let~$m(E)=\displaystyle{\int_{E}t^{\alpha}dxdt}$~and let~$u$~be a Lipschitz continuous function with compact support in~$\mathbb{\mathbb{R}}_+^{n+1}$.~Suppose~$\{x\in\mathbb{\mathbb{R}}_+^{n+1}:|u|>\lambda\}$~is finite for every positive~$\lambda$.~Then there exists a radial rearrangement~$u^{*}$~of~$u$~such that the following statements are valid.
\begin{enumerate}
\item[\rm{(i)}]$\mbox{For all}~\lambda,~m(\{|u|>\lambda\})=m(\{u^{*}>\lambda\}).$
\item[\rm{(ii)}]$u^{*}~\mbox{is a radially decreasing function}.$
\item[\rm{(iii)}]For~every~Young~function~$\Phi$~(that~is,~$\Phi$~maps~$[0,\infty)$~into~$[0,\infty)$,~vanishes~\\
at~0~and~it~is~convex~and~increasing):
$$\int_{\mathbb{\mathbb{R}}_{+}^{n+1}}\Phi(|\nabla u^{*}|)t^{\alpha}dxdt\leq\int_{\mathbb{\mathbb{R}}_{+}^{n+1}}\Phi(|\nabla u|)t^{\alpha}dxdt.$$
\item[\rm{(iv)}]$\mbox{If}~\Psi:~\mathbb{\mathbb{R}}^{+}\rightarrow \mathbb{\mathbb{R}}^{+}~\mbox{is a nondecreasing function,}~\mbox{then}$
$$\int_{\mathbb{\mathbb{R}}_{+}^{n+1}}\Psi(u)t^{\beta}dxdt=\int_{\mathbb{\mathbb{R}}_{+}^{n+1}}\Psi(u^{*})t^{\beta}dxdt.$$
\end{enumerate}
\end{lemma}
\begin{proof}[Proof of theorem \ref{theorem1.1}]
Using Lemma \ref{lemma2.2},~we can assume that~$u\in C_{c}^{\infty}(\mathbb{R}_{+}^{2})$~is radially decreasing with~$\rm{supp}(u)=\overline{B_{R}^{+}}\subset\mathbb{R}_{+}^{2}$~where~$R>0$.~Write~$b_{\alpha}=\frac{2+\alpha}{1+\alpha}$~as~$\alpha>-1$.~It yields that ~$b>1$.~Let~$\displaystyle{c_{\beta}=\int_{0}^{\pi}(\sin\theta)^{\beta}d\theta}$~and~$\displaystyle{c_{\alpha}=\int_{0}^{\pi}(\sin\theta)^{\alpha}d\theta}$,~we have
\begin{align}\label{2.1}
&\displaystyle{\int_{\Omega}\left(e^{a|u|^{b_{\alpha}}}-1\right)t^{\beta}dxdt}=\int_{0}^{\pi}(\sin\theta)^{\beta}d\theta\int_{0}^{R}\left(e^{a|u|^{b_{\alpha}}}-1\right)\rho^{1+\beta}d\rho=c_{\beta}\int_{0}^{R}\left(e^{a|u|^{b_{\alpha}}}-1\right)\rho^{1+\beta}d\rho
\end{align}
and
\begin{align}\label{2.2}
&\displaystyle{\int_{\Omega}|\nabla u|^{2+\alpha}t^{\alpha}dxdt}=\int_{0}^{\pi}(\sin\theta)^{\alpha}d\theta\int_{0}^{R}|u'(\rho)|^{2+\alpha}\rho^{1+\alpha}d\rho=c_{\alpha}\int_{0}^{R}|u'(\rho)|^{2+\alpha}\rho^{1+\alpha}d\rho,
\end{align}
where
\begin{align}\label{2.3}
&c_{\alpha}=\int_{0}^{\pi}(\sin\theta)^{\alpha}d\theta=2\int_{0}^{\frac{\pi}{2}}(\sin\theta)^{\alpha}d\theta=2B\left(\frac{1}{2},\frac{\alpha+1}{2}\right)=\frac{2\Gamma\left(\frac{\alpha+1}{2}\right)\Gamma\left(\frac{1}{2}\right)}{\Gamma(\frac{\alpha}{2}+1)}=2\pi^{\frac{1}{2}}\frac{\Gamma\left(\frac{\alpha+1}{2}\right)}{\Gamma\left(\frac{\alpha}{2}+1\right)}.
\end{align}
We define~$v(s)=Tu\left(Re^{-\frac{s}{2+\beta}}\right)$.~Then
\begin{align}\label{2.4}
&\quad\ \displaystyle{\int_{0}^{R}\left(e^{a|u|^{b_{\alpha}}}-1\right)\rho^{1+\beta}d\rho}&\notag\\
&=\int_{0}^{\infty}\left(e^{{a\left|u(Re^{-\frac{s}{2+\beta}})\right|}^{b_{\alpha}}}-1\right)\left(Re^{-\frac{s}{2+\beta}}\right)^{1+\beta}\frac{R}{2+\beta}e^{-\frac{s}{2+\beta}}ds\notag\\
&=\frac{R^{2+\beta}}{2+\beta}\int_{0}^{\infty}\left(e^{{\frac{a}{T^{b_{\alpha}}}}|v(s)|^{b_{\alpha}}}-1\right)e^{-\frac{s(2+\beta)}{2+\beta}}ds\notag\\
&=\frac{R^{2+\beta}}{2+\beta}\int_{0}^{\infty}\left(e^{\frac{a}{T^{b_{\alpha}}}|v(s)|^{b_{\alpha}}}-1\right)e^{-s}ds,
\end{align}
and
\begin{align}\label{2.5}
&\quad\ \displaystyle{\int_{0}^{R}|u'(\rho)|^{2+\alpha}\rho^{1+\alpha}d\rho}&\notag\\
&=\int_{0}^{\infty}|u'\left(Re^{\frac{-s}{2+\beta}}\right)|^{2+\alpha}\left(\frac{R}{{(2+\beta)}}e^{\frac{-s}{2+\beta}}\right)^{2+\alpha}\left(Re^{\frac{-s}{2+\beta}}\right)^{1+\alpha}\frac{R}{2+\beta}e^{\frac{-s}{2+\beta}}ds\notag\\
&=\left(\frac{2+\beta}{T}\right)^{2+\alpha}\frac{1}{2+\beta}\int_{0}^{\infty}|v'(s)|^{2+\alpha}\left(e^{\frac{s}{2+\beta}}\right)^{2+\alpha}\left(e^{\frac{-s}{2+\beta}}\right)^{2+\alpha}ds\notag\\
&=\left(\frac{2+\beta}{T}\right)^{2+\alpha}\frac{1}{2+\beta}\int_{0}^{\infty}|v'(s)|^{2+\alpha}ds.
\end{align}
Now,~we can select~$T$~such that~$(\frac{2+\beta}{T})^{2+\alpha}\frac{1}{2+\beta}c_{\alpha}=1$.~Then~$T=(2+\beta)(\frac{c_{\alpha}}{2+\beta})^{\frac{1}{2+\alpha}}$,~\begin{align}\label{2.6}T^{b_{\alpha}}=\left((2+\beta)^{\frac{1}{b_{\alpha}}}c_{\alpha}^{\frac{1}{2+\alpha}}\right)^{b_{\alpha}}=(2+\beta)c_{\alpha}^{\frac{1}{1+\alpha}}.\end{align}
Combining~\eqref{2.5}~and~\eqref{2.6},~we obtain
\begin{align*}
&\int_{\Omega}|\nabla u|^{2+\alpha}t^{\alpha}dxdt=c_{\alpha}\int_{0}^{R}|u'(\rho)|^{2+\alpha}\rho^{1+\alpha}d\rho&\\
&\quad\quad\quad\quad\quad\quad\quad~~=c_{\alpha}\left(\frac{2+\beta}{T}\right)^{2+\alpha}\frac{1}{2+\beta}\int_{0}^{\infty}|v'(s)|^{2+\alpha}ds&\\
&\quad\quad\quad\quad\quad\quad\quad~~=\int_{0}^{\infty}|v'(s)|^{2+\alpha}ds\leq1.
\end{align*}
Choosing
\begin{align*}
a(r,s)=\begin{cases}
1,&\text{0$\leq r\leq s$},\\0,&\text{$s<r$},\\
\end{cases}
\end{align*}
~$p'=b_{\alpha}$,~$\phi(r)=v'(r)$~and using Lemma \ref{lemma2.1},~we get
\begin{align*}
\sup\limits_{s>0}\int_{s}^{\infty}a(r,s)^{b_{\alpha}}dr=0<\infty,
\end{align*}
and
\begin{align*}
\int_{0}^{\infty}e^{-F(s)}ds=\int_{0}^{\infty}e^{\displaystyle{-s+\left(\int_{0}^{+\infty}a(r,s)\phi(r)ds\right)^{b_{\alpha}}}}ds,
\end{align*}
where
\begin{align*}
F(s)=s-\displaystyle{{\left(\int_{0}^{\infty}a(r,s)\phi(r)dr\right)}^{b_{\alpha}}}.
\end{align*}
By
\begin{align*}
\displaystyle{\int_{0}^{+\infty}a(r,s)\phi(r)dr}=\int_{0}^{s}\phi(r)dr=\int_{0}^{s}v'(r)dr=v'(s),
\end{align*}
there exists a constant~$c_{0}=c(\alpha,\beta)$~in~$\mathbb{R}$~such that
\begin{align}\label{2.7}
\displaystyle{\frac{1}{2+\beta}\int_{0}^{\infty}\left(e^{\frac{a_{\alpha,\beta}}{T^{b_{\alpha}}}|v(s)|^{b_{\alpha}}}-1\right)e^{-s}ds}
=\displaystyle{\frac{1}{2+\beta}\int_{0}^{\infty}\left(e^{|v(s)|^{b_{\alpha}}}-1\right)e^{-s}ds\leq c_{0}}.
\end{align}
Combining~\eqref{2.1},~\eqref{2.4},~\eqref{2.7}~and~$\frac{a}{T^{b}}\leq\frac{a_{\alpha,\beta}}{T^{b}}=1,$~we get
\begin{align}\label{2.8}
&\quad\ \displaystyle{\int_{\Omega}\left(e^{a{|u|}^{b_{\alpha}}}-1\right)t^{\beta}dxdt}&\notag\\
&=c_{\beta}\int_{0}^{R}\left(e^{a{|u|}^{b_{\alpha}}}-1\right)\rho^{1+\beta}d\rho\notag\\
&=\frac{c_{\beta}R^{2+\beta}}{2+\beta}\int_{0}^{+\infty}\left(e^{\frac{a}{T^{b_{\alpha}}}{|v(s)|}^{b_{\alpha}}}-1\right)e^{-s}ds\notag\\
&\leq c_{0}c_{\beta}R^{2+\beta}.
\end{align}A simple calculation yields that~$c_{\beta}=\displaystyle{\int_{0}^{\pi}\sin^{\beta}\theta d\theta=\int_{\partial B^{+}}t^{\beta}d\theta}$~and~$\frac{\displaystyle{\int_{B_{R}^{+}}t^{\beta}dxdt}}{\displaystyle{\int_{B^{+}}t^{\beta}dxdt}}=R^{2+\beta}$.~Then
\begin{align}\label{2.9}
&c_{0}c_{\beta}R^{2+\beta}=\displaystyle{c_{0}\int_{\partial B^{+}}t^{\beta}d\sigma R^{2+\beta}}=\displaystyle{c_{0}\int_{\partial B^{+}}t^{\beta}d\sigma\frac{\displaystyle{\int_{B_{R}^{+}}}t^{\beta}dxdt}{\displaystyle{\int_{B^{+}}t^{\beta}dxdt}}}.
\end{align}
By
\begin{align}\label{2.10}
\displaystyle{\int_{B^{+}}t^{\beta}}dxdt=\displaystyle{\int_{0}^{1}\int_{\partial B_{r}^{+}}t^{\beta}d\theta dr}=\displaystyle{\int_{0}^{1}r^{1+\beta}\left(\int_{\partial B^{+}}t^{\beta}d\sigma\right)dr}=\frac{1}{2+\beta}\displaystyle{\left(\int_{\partial B^{+}}t^{\beta}d\sigma\right)},
\end{align}
we get
\begin{align}\label{2.11}
&\displaystyle{c_{0}\int_{\partial B^{+}}t^{\beta}d\sigma\frac{\displaystyle{\int_{B_{R}^{+}}}t^{\beta}dxdt}{\displaystyle{\int_{B^{+}}t^{\beta}dxdt}}}=(2+\beta)c_{0}\displaystyle{\int_{B^{+}}t^{\beta}dxdt\frac{\displaystyle{\int_{B_{R}^{+}}t^{\beta}dxdt}}{\displaystyle{\int_{B^{+}}t^{\beta}dxdt}}}=(2+\beta)c_{0}\displaystyle{\int_{B_{R}^{+}}t^{\beta}dxdt}.
\end{align}
Then
\begin{align}\label{2.12}
\displaystyle{\int_{\Omega}\left(e^{a{|u|}^{b_{\alpha}}}-1\right)t^{\beta}dxdt\leq c_{0}m_{\beta}(\rm{supp}(u))(2+\beta)}.
\end{align}

It can be shown that the constant~$a_{\alpha,\beta}$~is sharp.~Because of the definition of Moser~sequence~$M_{n}$~which is analogous to~\cite{M1971},~we construct
\begin{align}\label{2.13}
M_{n}(x)=c_{\alpha}^{-\frac{1}{2+\alpha}}\begin{cases}
\text{$(\frac{n}{2+\beta})^{\frac{1}{b_{\alpha}}}$},&\text{${0\leq |(x,t)|\leq e^{-\frac{n}{2+\beta}}}$},\\
\text{$(\frac{2+\beta}{n})^{1-\frac{1}{b_{\alpha}}}\ln\frac{1}{|(x,t)|}$},&\text{${e^{-\frac{n}{2+\beta}}<|x,t|\leq 1}$},\\
0,&\text{${|x|\geq1}$}.
\end{cases}
\end{align}
By a simple calculation,~we obtain~$\displaystyle{\int_{\mathbb{\mathbb{\mathbb{R}}}_{+}^{2}}|\nabla M_{n}|^{2+\alpha}t^{\alpha}dxdt=1}$.~For~$a>a_{\alpha,\beta}$,~it can be assumed that~$\Omega=B_{+}$.~It follows from~\eqref{2.13}~that
\begin{align}\label{2.14}
&\ \quad\ \quad\quad\quad\quad\quad\quad\quad\quad\quad\ \displaystyle{\int_{\Omega}\left(e^{a|u|^{b_{\alpha}}}-1\right)t^{\beta}dxdt}&\notag\\
&\ \quad\ \quad\quad\quad\quad\quad\quad\quad\quad=c_{\beta}\int_{0}^{1}\left(e^{a|M_{n}|^{b_{\alpha}}}-1\right)\rho^{\beta+1}d\rho\notag\\
&\ \quad\ \quad\quad\quad\quad\quad\quad\quad\quad\geq c_{\beta}\int_{0}^{e^{\frac{-n}{2+\beta}}}\left(e^{{a}|\left(\frac{n}{2+\beta}\right)^{\frac{1}{b_{\alpha}}}\left(\frac{1}{c_{\alpha}}\right)^{\frac{1}{2+\alpha}}|^{b_{\alpha}}}-1\right)\rho^{\beta+1}d\rho\notag\\
&\ \quad\ \quad\quad\quad\quad\quad\quad\quad\quad\geq c_{\beta}\int_{0}^{e^{\frac{-n}{2+\beta}}}\left(e^{\frac{a}{a_{\alpha,\beta}}n}-1\right)\rho^{\beta+1}d\rho&\notag\\
&\ \quad\ \quad\quad\quad\quad\quad\quad\quad\quad=c_{\beta}\left(e^{\frac{a}{a_{\alpha,\beta}}n}-1\right)\int_{0}^{e^{\frac{-n}{2+\beta}}}\rho^{\beta+1}d\rho\notag\\
&\ \quad\ \quad\quad\quad\quad\quad\quad\quad\quad=c_{\beta}\frac{1}{2+\beta}\left(e^{\frac{an}{a_{\alpha,\beta}}}-1\right)e^{-n}\notag\\
&\ \quad\ \quad\quad\quad\quad\quad\quad\quad\quad=\frac{c_{\beta}}{2+\beta}\left(e^{\left(\frac{a}{a_{\alpha,\beta}}-1\right)n}-e^{-n}\right).
\end{align}
Therefore,~$\frac{c_{\beta}}{2+\beta}\left(e^{\left(\frac{a}{a_{\alpha,\beta}}-1\right)n}-e^{-n}\right)\rightarrow\infty$,~as~$n\rightarrow\infty$.
\end{proof}
In the following corollary,~we present the weighted Euclidean Onofri inequality on two-dimensional upper half space and write~$D_{\alpha}^{1,2+\alpha}(\Omega)$~by the weighted Sobolev space under the norm
\begin{align*}
\|u\|_{D_{\alpha}^{1,2+\alpha}(\Omega)}=\displaystyle{\left(\int_{\Omega}|\nabla u|^{2+\alpha}t^{\alpha}dxdt\right)^{\frac{1}{2+\alpha}}}.
\end{align*}
\begin{corollary}\label{corollary 2.3}
Let~$\alpha,\beta>-1$~and~$\zeta_{\alpha,\beta,u}=e^{(a_{\alpha,\beta}\frac{2+\alpha}{1+\alpha})^{-(1 +\alpha)}\frac{1}{2+\alpha}\|u\|_{D_{\alpha}^{1,2+\alpha}(\Omega)}^{^{2+\alpha}}}$.~Then there exists a constant~$c_{0}=c_{0}(\alpha,\beta)$~such that for each~$u\in C_c^\infty(\Omega)$~with~$\displaystyle{\int_{\Omega}|\nabla u|^{2+\alpha}t^{\alpha}dxdt}\leq1$,~we have
\begin{equation*}
\ln{\displaystyle{\int_{\Omega}\left(e^{u}-\zeta_{\alpha,\beta,u}\right)t^{\beta}dxdt}}\leq ln{\left((2+\beta)c_{0}m_{\beta}(\rm{supp}(u))\right)}+\left(a_{\alpha,\beta}\frac{2+\alpha}{1+\alpha}\right)^{-(1+\alpha)}\frac{1}{2+\alpha}\|u\|_{D_{\alpha}^{1,2+\alpha}(\Omega)}^{2+\alpha}
\end{equation*}
where\begin{equation*}
a_{\alpha,\beta}=(2+\beta)\left(2\pi^{\frac{1}{2}}\frac{\Gamma(\frac{\alpha+1}{2})}{\Gamma(\frac{\alpha}{2}+1)}\right)^{\frac{1}{1+\alpha}}
\end{equation*}
is the best constant.
\end{corollary}
\begin{proof}
By Young inequality with~$\epsilon$,~$ab\leq\epsilon a^{p}+c(\epsilon)b^{q}(a,b>0)$~where~$c(\epsilon)=(\epsilon p)^{-\frac{q}{p}}q^{-1}$,~we choose ~$a=\frac{u}{\|u\|_{D_{\alpha}^{1,2+\alpha}(\Omega)}},~b=\|u\|_{D_{\alpha}^{1,2+\alpha}(\Omega)}$,~$p=\frac{2+\alpha}{1+\alpha}$,~$q=2+\alpha,~c(\epsilon)=(a_{\alpha,\beta}\frac{2+\alpha}{1+\alpha})^{-(1 +\alpha)}\frac{1}{2+\alpha},~\epsilon=a_{\alpha,\beta}$~and then
\begin{align*}
&\ \ \quad\quad\quad\quad\quad\quad\quad\quad\quad\ \displaystyle{\int_{\Omega}(e^{u}-\zeta_{\alpha,\beta,u})t^{\beta}dxdt}&\\
&\ \quad\quad\quad\quad\quad\quad\quad\quad=\displaystyle{\int_{\Omega}\left(e^{\frac{u}{\|u\|_{D_{\alpha}^{1,2+\alpha}(\Omega)}}\|u\|_{D_{\alpha}^{1,2+\alpha}(\Omega)}}-e^{c(\epsilon)\|u\|_{D_{\alpha}^{1,2+\alpha}(\Omega)}^{^{2+\alpha}}}\right)t^{\beta}dxdt}\\
&\ \quad\ \quad\quad\quad\quad\quad\quad\quad\displaystyle{\leq\int_{\Omega}\left(e^{\epsilon\left(\frac{u}{\|u\|_{D_{\alpha}^{1,2+\alpha}(\Omega)}}\right)^{\frac{2+\alpha}{1+\alpha}}+c(\epsilon)\|u\|_{D_{\alpha}^{1,2+\alpha}(\Omega)}^{^{2+\alpha}}}-e^{c(\epsilon)\|u\|_{D_{\alpha}^{1,2+\alpha}(\Omega)}^{^{2+\alpha}}}\right)t^{\beta}dxdt}\\
&\ \quad\ \quad\quad\quad\quad\quad\quad\quad\displaystyle{\leq e^{c(\epsilon)\|u\|_{D_{\alpha}^{1,2+\alpha}(\Omega)}^{^{2+\alpha}}}\int_{\Omega}\left(e^{{\epsilon}\left(\frac{u}{\|u\|_{D_{\alpha}^{1,2+\alpha}(\Omega)}}\right)^{\frac{2+\alpha}{1+\alpha}}}-1\right)t^{\beta}dxdt}\\
&\ \quad\ \quad\quad\quad\quad\quad\quad\quad\displaystyle{\leq e^{c(\epsilon)\|u\|_{D_{\alpha}^{1,2+\alpha}(\Omega)}^{^{2+\alpha}}}\sup\limits_{\|u\|_{D_{\alpha}^{1,2+\alpha}(\Omega)}\leq1}\int_{\Omega}\left(e^{{\epsilon}|u|^{\frac{2+\alpha}{1+\alpha}}}-1\right)t^{\beta}dxdt}\\
&\ \quad\ \quad\quad\quad\quad\quad\quad\quad\displaystyle{\leq(2+\beta)c_{0}m_{\beta}(supp(u))e^{c(\epsilon)\|u\|_{D_{\alpha}^{1,2+\alpha}(\Omega)}^{^{2+\alpha}}}}.
\end{align*}
Here,~we use Theorem~\ref{theorem1.1} in the last step.~Therefore
\begin{align*}
\displaystyle{\int_{\Omega}\left(e^{u}-\zeta_{\alpha,\beta,u}\right)t^{\beta}dxdt\leq (2+\beta)c_{0}m_{\beta}(supp(u))e^{\left(a_{\alpha,\beta}\frac{2+\alpha}{1+\alpha}\right)^{-(1+\alpha)}\frac{1}{2+\alpha}\|u\|_{D_{\alpha}^{1,2+\alpha}(\Omega)}^{^{2+\alpha}}}}.
\end{align*}
After taking the logarithm,~we obtain the deserved conclusion.
\begin{equation*}
\ln{\displaystyle{\int_{\Omega}\left(e^{u}-\zeta_{\alpha,\beta,u}\right)t^{\beta}dxdt}}\leq ln{\left((2+\beta)c_{0}m_{\beta}(supp(u))\right)}+\left(a_{\alpha,\beta}\frac{2+\alpha}{1+\alpha}\right)^{-(1+\alpha)}\frac{1}{2+\alpha}\|u\|_{D_{\alpha}^{1,2+\alpha}(\Omega)}^{2+\alpha}.
\end{equation*}
\end{proof}
\begin{corollary}\label{corollary 2.4}
Let~$\alpha,\beta>-1$,~$\alpha<\beta$,~$b\leq b_{\alpha,\beta}$~and let~$B^{+}_{{R}_{0}}$~be the two-dimensional upper hemisphere centered at the origin with radius~${R}_{0}$~which satisfies~$c_{\alpha}^{\frac{1}{2+\alpha}}(2+\beta)^{\frac{1+\alpha}{2+\alpha}}u(R_{0})=1$.~If~$u\in C_{0}^{\infty}(\Omega)$~is a symmetric-decreasing rearrangement function and satisfies~$\displaystyle{\int_{B^{+}_{{R}_{0}}}|\nabla u|^{2+\alpha}t^{\alpha}}dxdt\leq1$.~Then we have
\begin{align*}
\displaystyle{\int_{B^{+}_{R_{0}}}\left(e^{b|u|^{\frac{2+\beta}{1+\beta}}}-1\right)t^{\beta}dxdt\leq (2+\beta)c_{0}m_{\beta}(B_{R_{0}}^{+})},
\end{align*}
where~$b_{\alpha,\beta}={(2+\beta)}^{\frac{(1+\alpha)(2+\beta)}{(2+\alpha)(1+\beta)}}{c_{\alpha}}^{\frac{2+\beta}{(2+\alpha)(1+\beta)}}.$~
\end{corollary}
\begin{proof}
Write~$b_{\beta}=\frac{2+\beta}{1+\beta}$.~Using~\eqref{2.1}~and~\eqref{2.2},~we obtain
\begin{equation*}
\displaystyle{\int_{B^{+}_{R_{0}}}e^{a|u|^{b_{\beta}}}t^{\beta}dxdt=c_{\beta}\int_{0}^{R_{0}}\left(e^{a|u|^{b_{\beta}}}-1\right)\rho^{1+\beta}d\rho},
\end{equation*}
and
\begin{equation*}
\displaystyle{\int_{B^{+}_{R_{0}}}|\nabla u|^{2+\alpha}t^{\alpha}dxdt=c_{\alpha}\int_{0}^{R_{0}}|u'(\rho)|^{2+\alpha}\rho^{1+\alpha}d\rho}.~
\end{equation*}
We define~$v(s)=Tu(R_{0}e^{-\frac{s}{2+\beta}})$.~Then
\begin{align*}
\displaystyle{\int_{0}^{R_{0}}\left(e^{a|u|^{b_{\beta}}}-1\right)\rho^{1+\beta}d\rho}=\displaystyle{\frac{{R_{0}}^{2+\beta}}{2+\beta}\int^{\infty}_{0}\left( e^{\frac{a}{T^{b_{\beta}}}|v(s)|^{b_{\beta}}}-1\right)e^{-s}ds}.~
\end{align*}
By choosing
\begin{align*}
a(r,s)=\begin{cases}
1,&\text{0$\leq r\leq s$},\\0,&\text{$s<r$},\\
\end{cases}
\end{align*}
~$p'=b_{\alpha}$,~$\phi(r)=v'(r)$~and using Lemma~\ref{lemma2.1},~we obtain
\begin{align*}
\sup\limits_{s>0}\int_{s}^{\infty}a(r,s)^{b_{\alpha}}dr=0<\infty,
\end{align*}
and
\begin{align*}
\int_{0}^{\infty}e^{-F(s)}ds=\int_{0}^{\infty}e^{\displaystyle{-s+\left(\int_{0}^{+\infty}a(r,s)\phi(r)ds\right)^{b_{\alpha}}}}ds,
\end{align*}
where
\begin{align*}
F(s)=s-\displaystyle{{\left(\int_{0}^{\infty}a(r,s)\phi(r)dr\right)}^{b_{\alpha}}}.
\end{align*}
In view of
\begin{align*}
\displaystyle{\int_{0}^{+\infty}a(r,s)\phi(r)dr}=\int_{0}^{s}\phi(r)dr=\int_{0}^{s}v'(r)dr=v'(s),
\end{align*}
we find that there exists a constant~$c_{0}=c(\alpha,\beta)$~such that
\begin{align}\label{2.15}
\displaystyle{\frac{1}{2+\beta}\int_{0}^{\infty}\left(e^{|v(s)|^{b_{\alpha}}}-1\right)e^{-s}ds\leq c_{0}}.
\end{align}
Because~$\alpha<\beta$,~$v(s)={((2+\beta){c_{\alpha}}^{\frac{1}{1+\alpha}})}^{\frac{1}{b}}u(R_{0}e^{-\frac{s}{2+\beta}})$,~$R_{0}e^{-\frac{s}{2+\beta}}\leq R_{0}$~and~$v$~is an increasing function,~we obtain~$v(s)>1$~and~$\frac{2+\alpha}{1+\alpha}>\frac{2+\beta}{1+\beta}$.~By choosing~$b_{\alpha,\beta}=T^{b_{\beta}}$,~we have
\begin{align*}
&\quad\ \displaystyle{\int_{B^{+}_{R_{0}}}\left(e^{a{|u|}^{b_{\beta}}}-1\right)t^{\beta}dxdt}&\\
&=\displaystyle{c_{\beta}\int_{0}^{R_{0}}\left(e^{a{|u|}^{b_{\beta}}}-1\right)\rho^{1+\beta}d\rho}\\
&=\displaystyle{\frac{c_{\beta}R_{0}^{2+\beta}}{2+\beta}\int_{0}^{\infty}\left(e^{\frac{a}{T^{b_{\beta}}}{|v(s)|}^{b\beta}}-1\right)e^{-s}ds}&\\
&\displaystyle{\leq\frac{c_{\beta}R_{0}^{2+\beta}}{2+\beta}\int_{0}^{\infty}\left(e^{{|v(s)|}^{b_{\beta}}}-1\right)e^{-s}ds}\\
&\displaystyle{\leq\frac{c_{\beta}R_{0}^{2+\beta}}{2+\beta}\int_{0}^{\infty}\left(e^{{|v(s)|}^{b_{\alpha}}}-1\right)e^{-s}ds}\\
&=c_{\beta}R_{0}^{2+\beta}c_{0}
\end{align*}
~where~$a\leq b_{\alpha,\beta}$.~It follows from~\eqref{2.9}~and~\eqref{2.11}~that
\begin{align*}
c_{\beta}R_{0}^{2+\beta}c_{0}=(2+\beta)c_{0}\displaystyle{\int_{B_{R_{0}}^{+}}t^{\beta}dxdt}.
\end{align*}
Therefore,~
\begin{align*}
\displaystyle{\int_{B^{+}_{R_{0}}}\left(e^{a{|u|}^{b_{\beta}}}-1\right)t^{\beta}dxdt}
\leq(2+\beta)c_{0}\displaystyle{\int_{B_{R_{0}}^{+}}t^{\beta}dxdt}=c_{0}R_{0}^{2+\beta}c_{\beta}.
\end{align*}
\end{proof}
\begin{remark}
We conjecture that this Trudinger type inequality on the upper half space for two dimensions has a optimal constant.~But it is a challenge to obtain this by using classical methods.~We leave it as a future problem.
\end{remark}

\section{Concentration-compactness principle}
\begin{lemma}\rm{(Lions 1985~\cite{L1985})}\label{lemma 3.1}.~Suppose~$\{u_{n}\}\subset W_{0}^{1,N}(\Omega)$~is a sequence satisfies that~$\|\nabla u_{n}\|_{n}\leq1$,~$u_{n}\rightarrow u$~and~$|\nabla u(x)|^{n}dx\rightarrow d\mu$~weakly.~Then there exists either~$\alpha>0$~such that ~$e^{(\alpha_{n}+\alpha){|u_{n}|}^{\frac{n}{n-1}}}$~is bounded in~$L^{1}(\Omega)$~for~$u\not\equiv0$~and thus as~$n\rightarrow\infty$,~
\begin{align*}
e^{\alpha_{n}{|u_{n}|}^{\frac{n}{n-1}}}\mathop{\rightarrow}e^{\alpha_{n}{|u|}^{\frac{n}{n-1}}}~in~L^{1}(\Omega)
\end{align*}
or~$u=\delta_{x_{0}}$~for some~$x_{0}\in\overline{\Omega}$~and~$u_{n}\mathop{\rightarrow}0$,~$e^{(\alpha_{n}{|u_{n}|}^{\frac{n}{n-1}})}\mathop{\rightarrow}c\delta_{x_{0}}$~for some~$c\geq0$.~
\end{lemma}
\begin{definition}$(\text{Concentration})$\label{definition 3.2}
Let~$B^{+}_{\delta}$~be the two-dimensional unit upper hemisphere centered at the origin with radius~$\delta$.~If~$\|u\|_{D_{0,\alpha,rad}^{1,2+\alpha}}\leq1$~and for any~$0<\delta<1$,~$\int_{B^{+}\backslash B^{+}_{\delta}}|\nabla u_{k}|^{2+\alpha}t^{\alpha}\rightarrow0$,~then the function sequence~$u_{k}\in D_{0,\alpha,rad}^{1,2+\alpha}$~is said to be concentrated at the origin denoted by~$|\nabla u_{k}|^{2+\alpha}t^{\alpha}\hookrightarrow\delta_{0}$.~
\end{definition}
\begin{lemma}\label{lemma 3.3}
Suppose~$\tilde{f}_{m}$~is a sequence in~${D_{0,\alpha,rad}^{1,2+\alpha}(B^{+})}$~such that~$\tilde{f}_{m}\rightharpoonup \tilde{f}$~as~$m\rightarrow+\infty$,~$\|\tilde{f}_{m}\|_{D_{0,\alpha,rad}^{1,2+\alpha}(B^{+})}\leq1$.~Subsequently for any subsequence,~either
\begin{enumerate}
\item[\rm{(i)}]$J(\tilde{f}_{m})\rightarrow J(\tilde{f}),~as~m\rightarrow+\infty$
\end{enumerate}
or
\begin{enumerate}
\item[\rm{(ii)}]$\tilde{f}_{m}~\text{concentrates at x=0}.$
\end{enumerate}
\end{lemma}
\begin{proof}
Let~$|x|=r=e^{-\frac{s}{2+\beta}}$,~$f(s)=(2+\beta){(\frac{c_{\alpha}}{2+\beta})}^{\frac{1}{2+\alpha}}\tilde{f}(x)$.~Then
\begin{align*}
{|\nabla\tilde{f}|}^{2+\alpha}={|\tilde{f}_{r}|}^{2+\alpha}={\left|\frac{1}{2+\beta}{\left(\frac{2+\beta}{c_{\alpha}}\right)}^{\frac{1}{2+\alpha}}f_{s}\cdot(2+\beta)(-e^{\frac{s}{2+\beta}})\right|}^{2+\alpha}=\frac{2+\beta}{c_{\alpha}}{|f_{s}|}^{2+\alpha}e^{\frac{s(2+\alpha)}{2+\beta}},
\end{align*}
~$dr=e^{-\frac{s}{2+\beta}}\cdot(-\frac{1}{2+\beta})ds$.~
Therefore,~
\begin{align*}
&\quad\ \displaystyle{c_{\alpha}\int_{\delta}^{1}{|\tilde{f}_{r}|}^{2+\alpha}r^{1+\alpha}dr}&\\
&=c_{\alpha}\int_{-(2+\beta)\ln\delta}^{0}\frac{2+\beta}{c_{\alpha}}{|f_{s}|}^{2+\alpha}e^{\frac{s(2+\alpha)}{2+\beta}}e^{-\frac{s(1+\alpha)}{2+\beta}}e^{-\frac{s}{2+\beta}}\left(-\frac{1}{2+\beta}\right)ds\\
&=c_{\alpha}\int_{0}^{-(2+\beta)\ln\delta}\frac{2+\beta}{c_{\alpha}}\frac{1}{2+\beta}{|f_{s}|}^{2+\alpha}ds\\
&=\int_{0}^{-(2+\beta)\ln\delta}{|f_{s}|}^{2+\alpha}ds,
\end{align*}
and
\begin{align*}
\displaystyle{\int_{B^{+}/B_{\delta}}{|\nabla\tilde{f}|}^{2+\alpha}t^{\alpha}dxdt}=\displaystyle{c_{\alpha}\int_{\delta}^{1}{|\tilde{f}_{r}|}^{2+\alpha}r^{1+\alpha}dr}=\int_{0}^{-(2+\beta)\ln\delta}{|f_{s}|}^{2+\alpha}ds.
\end{align*}
Choosing~$A=-(2+\beta)\ln\delta$,~we obtain~$\delta=e^{-\frac{A}{2+\beta}}$.~Then
\begin{align*}
\displaystyle{\int_{B^{+}/B_{\delta}}{|\nabla \tilde{f}|}^{2+\alpha}t^{\alpha}dxdt}=\displaystyle{\int_{0}^{A}{|f_{s}|}^{2+\alpha}ds}.
\end{align*}
We want to prove that~$(\mathbf{i})$~is true if we suppose that~$(\mathbf{ii})$~is false.

$(\mathbf{ii})$~falses that means the sequence does not concentrate at ~$x=0$.~Then there exists some~$A>0$~and~$0<k<1$~with
\begin{align*}
\displaystyle{\int_{0}^{A}{|f_{m}'|}^{2+\alpha}ds\geq k}.
\end{align*}

For~$m\geq m_{0}$~and~$s\geq A$,~
using the~$\rm{H\ddot{o}lder}$~inequality,~we obtain
\begin{align}\label{3.1}
&f_{m}(s)-f_{m}(A)=\displaystyle{\int_{A}^{s}f_{m}'(\tau)d\tau}&\notag\\
&\quad\quad\quad\quad\quad\quad=\displaystyle{\int_{A}^{s}f_{m}'(\tau)\cdot1d\tau}\notag\\
&\quad\quad\quad\quad\quad\quad\leq{\left(\displaystyle{\int_{A}^{s}{|f'_{m}|}^{2+\alpha}}d\tau\right)}^{\frac{1}{2+\alpha}}{\left(\displaystyle{\int_{A}^{s}{1}^{\frac{2+\alpha}{1+\alpha}}}d\tau\right)}^{\frac{1+\alpha}{2+\alpha}}\notag\\
&\quad\quad\quad\quad\quad\quad={{(s-A)}^{\frac{1+\alpha}{2+\alpha}}\left(\displaystyle{\int_{A}^{s}{|f'_{m}|}^{2+\alpha}d\tau}\right)}^{\frac{1}{2+\alpha}}&\notag\\
&\quad\quad\quad\quad\quad\quad\leq{(s-A)}^{\frac{1+\alpha}{2+\alpha}}{(1-k)}^{\frac{1}{2+\alpha}}\notag\\
&\quad\quad\quad\quad\quad\quad\leq s^{\frac{1+\alpha}{2+\alpha}}{(1-k)}^{\frac{1}{2+\alpha}}.
\end{align}
 Since the value of~$f_{m}(0)$~corresponds to~$\tilde{f}_{m}(1)$~and~$\tilde{f}_{m}$~has compact support,~we obtain~$f_{m}(0)=0$.~Using~$\rm{H\ddot{o}lder}$~inequality,~we obtain
\begin{align*}
&f_{m}(A)=\displaystyle{\int_{0}^{A}f_{m}'(\tau)d\tau}\\
&\quad\quad\quad\displaystyle{\leq{\left(\int_{0}^{A}{|f_{m}'|}^{2+\alpha}d\tau\right)}^{\frac{1}{2+\alpha}}{{\left(\int_{0}^{A}1d\tau\right)}^{\frac{1+\alpha}{2+\alpha}}}}~\\
&\quad\quad\quad\displaystyle{={A^{\frac{1+\alpha}{2+\alpha}}\left(\int_{0}^{A}{|f_{m}'|}^{2+\alpha}d\tau\right)}^{\frac{1}{2+\alpha}}}.
\end{align*}
By constraint condition,~we obtain
\begin{align*}
\displaystyle{\int_{0}^{A}{|f_{m}'(\tau)|}^{2+\alpha}d\tau\leq1}.~
\end{align*}
Therefore,~
\begin{align}\label{3.2}
f_{m}(A)\leq A^{\frac{1+\alpha}{2+\alpha}}.~
\end{align}

In the last step,~the following inequality is used.~This can be easily obtained by using L'Hopital's rule:~if~$0<\gamma<\mu$,~$p>1$,~then for a sufficiently large~$y\in\mathbb{R}$,~one has
\begin{align*}
{(1+\gamma y)}^{p}\leq1+\mu^{p}y^{p}.
\end{align*}
Where~$y=s^{\frac{1+\alpha}{2+\alpha}}$,~$p=\frac{2+\alpha}{1+\alpha}$,~$\gamma=A^{-\frac{1+\alpha}{2+\alpha}}{(1-k)}^{\frac{1}{2+\alpha}}$~and~$\mu=A^{-\frac{1+\alpha}{2+\alpha}}{(1-\frac{k}{2})}^{\frac{1}{2+\alpha}}$.~\\
By~\eqref{3.1}~and~\eqref{3.2},~we obtain
\begin{align}\label{3.3}
&{f_{m}^{\frac{2+\alpha}{1+\alpha}}(s)}\leq\displaystyle{{\left({(1-k)}^{\frac{1}{2+\alpha}}s^{\frac{1+\alpha}{2+\alpha}}+A^{\frac{1+\alpha}{2+\alpha}}\right)}^{\frac{2+\alpha}{1+\alpha}}}\notag\\
&\quad\quad\quad\ ={\left(1+A^{-\frac{1+\alpha}{2+\alpha}}{(1-k)}^{\frac{1}{2+\alpha}}s^{\frac{1+\alpha}{2+\alpha}}\right)}^{\frac{2+\alpha}{1+\alpha}}{\left(A^{\frac{1+\alpha}{2+\alpha}}\right)}^{\frac{2+\alpha}{1+\alpha}}\notag\\
&\quad\quad\quad\ \leq A+s{\left(1-\frac{k}{2}\right)}^{\frac{1}{1+\alpha}}.
\end{align}
Therefore~for a sufficiently large~$s$,~${f_{m}^{\frac{2+\alpha}{1+\alpha}}(s)}\leq A+s{(1-\frac{k}{2})}^{\frac{1}{1+\alpha}}$.~Because~$a_{\alpha,\beta}=(2+\beta){c_{\alpha}}^{\frac{1}{1+\alpha}}$,~we obtain
\begin{align}\label{3.4}
&\quad\ J(\tilde{f}_{m})&\notag\\
&=\frac{1}{(2+\beta)c_{0}m_{\beta}(B^{+})}\displaystyle{\int_{B^{+}}\left(e^{a_{\alpha,\beta}{|\tilde{f}_{m}|}^{\frac{2+\alpha}{1+\alpha}}}-1\right)t^{\beta}}dxdt\notag\\
&=\frac{1}{(2+\beta)c_{0}m_{\beta}(B^{+})}\displaystyle{\int_{0}^{\pi}\int_{0}^{1}\left(e^{a_{\alpha,\beta}{|\tilde{f}_{m}|}^{\frac{2+\alpha}{1+\alpha}}}-1\right){(r\sin\theta)}^{\beta}r}drd\theta&\notag\\
&=\frac{1}{(2+\beta)c_{0}m_{\beta}(B^{+})}\displaystyle{\int_{0}^{\pi}\sin^{\beta}\theta d\theta\int_{0}^{1}\left(e^{a_{\alpha,\beta}{|\tilde{f}_{m}|}^{\frac{2+\alpha}{1+\alpha}}}-1\right)r^{\beta+1}}dr\notag\\
&=\frac{1}{(2+\beta)c_{0}m_{\beta}(B^{+})}\displaystyle{\int_{0}^{\pi}}\sin^{\beta}\theta d\theta\int_{0}^{\infty}\left(e^{a_{\alpha,\beta}{|f_{m}{(2+\beta)}^{-1}{(\frac{2+\beta}{c_{\alpha}})}^{\frac{1}{2+\alpha}}|}^{\frac{2+\alpha}{1+\alpha}}}-1\right){\left(e^{-\frac{s}{2+\beta}}\right)}^{1+\beta}\frac{1}{2+\beta}e^{-\frac{s}{2+\beta}}ds\notag\\
&
=\frac{1}{{(2+\beta)}^{2}c_{0}m_{\beta}(B^{+})}\displaystyle{\int_{0}^{\pi}\sin^{\beta}\theta d\theta\int_{0}^{\infty}\left(e^{a_{\alpha,\beta}{|f_{m}|}^{\frac{2+\alpha}{1+\alpha}}{(2+\beta)}^{-1}{c_{\alpha}}^{-\frac{1}{1+\alpha}}}-1\right)e^{-s}}ds\notag\\
&=\frac{c_{\beta}}{{(2+\beta)}^{2}c_{0}m_{\beta}(B^{+})}\int_{0}^{\infty}\left(e^{{|f_{m}|}^{\frac{2+\alpha}{1+\alpha}}}-1\right)e^{-s}ds.~
\end{align}
Let~$I(f_{m})=\displaystyle{\int_{0}^{\infty}\left(e^{{|f_{m}|}^{\frac{2+\alpha}{1+\alpha}}}-1\right)e^{-s}ds}$.~Now splitting~$I(f_{m})=I_{1}(f_{m})+I_{2}(f_{m})-1$,~we have
\begin{align*}
&I(f_{m})=I_{1}(f_{m})+I_{2}(f_{m})-1\\
&\ \quad\quad=\displaystyle{\int_{0}^{n}\left(e^{{|f_{m}|}^{\frac{2+\alpha}{1+\alpha}}-s}\right)ds+\int_{n}^{\infty}\left(e^{{|f_{m}|}^{\frac{2+\alpha}{1+\alpha}}-s}\right)ds}-1.
\end{align*}
Using~$f_{m}(s)\leq s^{\frac{1+\alpha}{2+\alpha}}$~and the dominated convergence theorem for~$I_{1}(w)$,~we obtain~$I_{1}(f_{m})\rightarrow I_{1}(f)$.~By~\eqref{3.3},~we have
\begin{align*}
I_{2}(f_{m})=\displaystyle{\int_{n}^{\infty}e^{{|f_{m}|}^{\frac{2+\alpha}{1+\alpha}}-s}ds}\leq\displaystyle{\int_{n}^{\infty}e^{-s+A+s{\left(1-\frac{k}{2}\right)}^{\frac{1}{1+\alpha}}}ds}=e^{A}\displaystyle{\int_{n}^{\infty}e^{\left({\left(1-\frac{k}{2}\right)}^{\frac{1}{1+\alpha}}-1\right)s}ds}.
\end{align*}
Because~$0<k<1$,~${(1-\frac{k}{2})}^{\frac{1}{1+\alpha}}-1$~is negative,~we observe that~$I_{2}(f_{m})$~is arbitrarily small when~$n$~is sufficiently large.~We know that~$f_{m}$~converges pointwise to~$f$~as~$m\rightarrow+\infty$~.~Therefore~$I(f_{m})\rightarrow I(f)$~implies that~$J(\tilde{f}_{m})\rightarrow J(\tilde{f})$.~
\end{proof}

\section{Extremal function exists under dynamic change}
First,~we introduce a new idea about the existence of extremal functions for sharp Moser-Trudinger type inequalities under dynamic change.~This means the existence of the extreme function of the inequality is affected by parameter changes in the constraint condition.
\begin{lemma}[Chang 1986~\cite{C1986}]\label{lemma 4.1}~For each~$c>0$,~$\delta_{0}>0$~and~$A_{\delta_{0}}=\{\phi\in C^{1}(0,+\infty)|\phi(0)=0,~\displaystyle{\int_{0}^{\infty}}{|\phi'|}^{2}dt\leq\delta_{0}\}$,~we have
\begin{equation*}
\sup\limits_{\phi\in\Lambda_{\delta_{0}}}\displaystyle{\int^{\infty}_{0}e^{c\phi(t)-t}dt\leq e^{\frac{c^{2}\delta_{0}}{4}+1}}.
\end{equation*}
\end{lemma}
\begin{lemma}\label{lemma 4.2}
For~$\alpha>-1$,~$\varsigma>-1$~and any two functions~$u$~and~$v$~whose domains and ranges are both in the real number field,~we choose~$u=v+L$~where~$L$~is any constant,~$c_{\varsigma,\alpha}={\left(1-{(1+\varsigma)}^{-1-\alpha}\right)}^{-\frac{1}{1+\alpha}}$,~then
\begin{align*}
u^{\frac{2+\alpha}{1+\alpha}}\leq(1+\varsigma)v^{\frac{2+\alpha}{1+\alpha}}+c(\varsigma,\alpha)L^{\frac{2+\alpha}{1+\alpha}}.~ \end{align*}
\end{lemma}
\begin{proof}
Let~$C$~be any nonzero constant and let~$a_{1}$,~$b_{1}$,~$a_{2}$,~$b_{2}$~be real numbers.~Then
\begin{align*}
&u=v+L=\frac{1}{{(1+\varsigma)}^{\frac{1}{q}}}\left({(1+\varsigma)}^{\frac{1}{q}}v\right)+\frac{1}{c^{\frac{1}{q}}}\left(c^{\frac{1}{q}}L\right).
\end{align*}
Using~H$\rm{\ddot{o}}$lder~inequality:
\begin{align*}
a_{1}b_{1}+a_{2}b_{2}\leq{\left({a_{1}}^{p}+{a_{2}}^{p}\right)}^{\frac{1}{p}}{\left(b_{1}^{q}+{b_{2}^{q}}\right)}^{\frac{1}{q}}
\end{align*}
where~$p=2+\alpha$,~$q=\frac{2+\alpha}{1+\alpha}$,~$c=c(\varsigma,\alpha)={\left(1-{(1+\varsigma)}^{-1-\alpha}\right)}^{-\frac{1}{1+\alpha}}$,~we have
\begin{align*}
&\quad\ \frac{1}{{(1+\varsigma)}^{\frac{1}{q}}}\left({(1+\varsigma)}^{\frac{1}{q}}v\right)+\frac{1}{c^{\frac{1}{q}}}\left(c^{\frac{1}{q}}L\right)\\
&\leq{\left({\left(\frac{1}{{(1+\varsigma)}^{\frac{1}{q}}}\right)}^{p}+{\left(\frac{1}{c^{\frac{1}{q}}}\right)}^{p}\right)}^{\frac{1}{p}}{\left({\left({(1+\varsigma)}^{\frac{1}{q}}v\right)}^{q}+{\left(c^{\frac{1}{q}}L\right)}^{q}\right)}^{\frac{1}{q}}\\
&={\left({\left(\frac{1}{{(1+\varsigma)}^{\frac{1+\alpha}{2+\alpha}}}\right)}^{2+\alpha}+{\left(\frac{1}{c^{\frac{1+\alpha}{2+\alpha}}}\right)}^{2+\alpha}\right)}^{\frac{1}{2+\alpha}}{\left({({(1+\varsigma)}^{\frac{1+\alpha}{2+\alpha}}v)}^{\frac{2+\alpha}{1+\alpha}}+{\left(c^{\frac{1+\alpha}{2+\alpha}}L\right)}^{\frac{2+\alpha}{1+\alpha}}\right)}^{\frac{1+\alpha}{2+\alpha}}.
\end{align*}
Therefore,~
\begin{align*}
u^{\frac{2+\alpha}{1+\alpha}}\leq(1+\varsigma)v^{\frac{2+\alpha}{1+\alpha}}+c(\varsigma,\alpha)L^{\frac{2+\alpha}{1+\alpha}}.
\end{align*}
\end{proof}
Write~$c_{1,\alpha}=\displaystyle{\int_{0}^{\pi}(\sin\theta)^{-\frac{\alpha}{1+\alpha}}d\theta}$,~$\Gamma(\phi)=\displaystyle{{\left(\int_{0}^{+\infty}{|\phi'(s)|}^{2+\alpha}ds\right)}^{\frac{1}{2+\alpha}}}$.~For ~$\delta\in(0,1]$,~$\tilde{\Lambda}_{\delta}:=\{\phi\in C^{1}(0,~\infty)|\phi(0)=0,~\Gamma(\phi)\leq\delta\}$.~The definition here is derived by studying the ideas in~\cite{C1986}.
\begin{lemma}\label{lemma 4.3}
For any~$a>0$,~$\varsigma>0$,~if~$\displaystyle{{\left(\int_{0}^{\infty}{|\phi'(s)|}^{\frac{2+\alpha}{1+\alpha}}ds\right)}^{\frac{1+\alpha}{2+\alpha}}\leq c_{1}}$,~$1-(1+\varsigma){\delta}>0$,~then
\begin{equation*}
\sup\limits_{\phi\in\tilde{\Lambda}_{\delta}}\displaystyle{\int_{\frac{c_{1}^{2}(1+\alpha)a}{16}}^{\infty}e^{\phi(s)^{\frac{2+\alpha}{1+\alpha}}-s}ds}\leq \frac{1}{(1-(1+\varsigma){\delta})}e^{c_{\varsigma,\alpha}\frac{c_{1}\left(1+\alpha\right)}{4}\phi^{\frac{2+\alpha}{1+\alpha}}\left(\frac{c_{1}a}{4}\right)-\frac{c_{1}^{2}\left(1+\alpha\right)a}{16}}
\end{equation*}
\begin{equation*}
^{+\frac{c_{1}{\delta}}{1-(1+\varsigma){\delta}}\frac{\phi\left(\frac{c_{1}a}{4}\right)}{4}+1}.
\end{equation*}
\begin{proof}
Let~$r=e^{-\frac{s}{2+\beta}}$,~$\phi(s)=(2+\beta){\left(\frac{c_{\alpha}}{2+\beta}\right)}^{\frac{1}{2+\alpha}}u(x)$,~$u(x)=\frac{1}{2+\beta}{\left(\frac{2+\beta}{c_{\alpha}}\right)}^{\frac{1}{2+\alpha}}\phi(s)$~and
~$\frac{du}{dr}=\frac{du}{ds}\frac{ds}{dr}=-{\left(\frac{2+\beta}{c_{\alpha}}\right)}^{\frac{1}{2+\alpha}}\phi'(s)e^{\frac{s}{2+\beta}}$.~We obtain
\begin{align*}
\displaystyle{\int_{B^{+}}{|\nabla u|}^{2+\alpha}t^{\alpha}dxdt=\int_{0}^{\infty}{|\phi'|}^{2+\alpha}ds}.
\end{align*}
We change the variable~$z=s-\frac{{c_{1}}^{2}(1+\alpha)a}{16}$,~$\phi(s)=\psi(z)+{\left(\frac{c_{1}(1+\alpha)}{4}\right)}^{\frac{1+\alpha}{2+\alpha}}\phi(\frac{c_{1}a}{4})$.~\\
Since~$\Gamma(\phi)\leq\delta$,
using the~$\rm{H\ddot{o}lder}$~inequality,~we obtain
\begin{align*}
&\displaystyle{\int_{0}^{\infty}|\psi'|^{2}dz}=\displaystyle{\int_{a}^{\infty}|\phi'(s)|^{2}ds}&\\
&\quad\quad\quad\quad\quad\leq\displaystyle{{\left(\int_{a}^{\infty}{|\phi'(s)|}^{2+\alpha}ds\right)}^{\frac{1}{2+\alpha}}{\left(\int_{a}^{\infty}{|\phi'(s)|}^{\frac{2+\alpha}{1+\alpha}}ds\right)}^{\frac{1+\alpha}{2+\alpha}}}\\
&\quad\quad\quad\quad\quad\leq c_{1}{\left(\int_{0}^{\infty}{|\phi'(s)|}^{2+\alpha}ds\right)}^{\frac{1}{2+\alpha}}\\
&\quad\quad\quad\quad\quad\leq c_{1}{\delta}.
\end{align*}
Let~$v_{k}=u_{k}-L$.~Using Lemma~\ref{lemma 4.2},~we obtain
\begin{align*}
{(v_{k}+L)}^{\frac{2+\alpha}{1+\alpha}}\leq(1+\varsigma)v_{k}^{\frac{2+\alpha}{1+\alpha}}+\left(c_{\varsigma,\alpha}^{\frac{1+\alpha}{2+\alpha}}L\right)^{\frac{2+\alpha}{1+\alpha}},~\forall\varsigma>0,
\end{align*}
and
\begin{align*}
&\quad\ [\psi(z)+{\left(\frac{c_{1}(1+\alpha)}{4}\right)}^{\frac{1+\alpha}{2+\alpha}}\phi\left(\frac{c_{1}a}{4}\right)]^{\frac{2+\alpha}{1+\alpha}}\\
&\leq(1+\varsigma)\psi(z)^{\frac{2+\alpha}{1+\alpha}}+{\left({c_{\varsigma,\alpha}}^{\frac{1+\alpha}{2+\alpha}}{\left(\frac{c_{1}(1+\alpha)}{4}\right)}^{\frac{1+\alpha}{2+\alpha}}\phi\left(\frac{c_{1}a}{4}\right)\right)}^{\frac{2+\alpha}{1+\alpha}}\\
&\leq(1+\varsigma)\psi(z)^{\frac{2+\alpha}{1+\alpha}}+{\left({c_{\varsigma,\alpha}}^{\frac{1+\alpha}{2+\alpha}}{\left(\frac{c_{1}(1+\alpha)}{4}\right)}^{\frac{1+\alpha}{2+\alpha}}\phi\left(\frac{c_{1}a}{4}\right)\right)}^{\frac{2+\alpha}{1+\alpha}}+\psi(z){\phi^{\frac{1}{2}}\left(\frac{c_{1}a}{4}\right)}.
\end{align*}
Therefore,~
\begin{align*}
&\quad\ \displaystyle{\int_{\frac{{c_{1}}^{2}(1+\alpha)a}{16}}^{\infty}e^{{\phi(s)}^{\frac{2+\alpha}{1+\alpha}}-s}ds}\\
&=\displaystyle{\int_{0}^{\infty}e^{[\psi(z)+{\left(\frac{c_{1}(1+\alpha)}{4}\right)}^{\frac{1+\alpha}{2+\alpha}}\phi\left(\frac{c_{1}a}{4}\right)]^{\frac{2+\alpha}{1+\alpha}}-z-\frac{c^{2}_{1}(1+\alpha)}{16}a}dz}\\
&\displaystyle{\leq e^{-\frac{{c_{1}}^{2}(1+\alpha)a}{16}}\int_{0}^{\infty}e^{(1+\varsigma)\psi(z)^{\frac{2+\alpha}{1+\alpha}}+{\left({c_{\varsigma,\alpha}}^{\frac{1+\alpha}{2+\alpha}}{\left(\frac{c_{1}(1+\alpha)}{4}\right)}^{\frac{1+\alpha}{2+\alpha}}\phi\left(\frac{c_{1}a}{4}\right)\right)}^{\frac{2+\alpha}{1+\alpha}}+\psi(z){\phi^{\frac{1}{2}}\left(\frac{c_{1}a}{4}\right)}-z}dz}\\
&=e^{-\frac{{c_{1}}^{2}(1+\alpha)a}{16}}e^{c_{\varsigma,\alpha}\frac{c_{1}(1+\alpha)}{4}\phi^{\frac{2+\alpha}{1+\alpha}}\left(\frac{c_{1}a}{4}\right)}\displaystyle{\int_{0}^{\infty}e^{\psi(z)\phi^{\frac{1}{2}}\left(\frac{c_{1}a}{4}\right)-(1-(1+\varsigma)\delta)z}dz}.
\end{align*}
Let~$y=[1-(1+\varsigma){\delta}]z$~and~$\Theta(y)=\psi(z)$.
\begin{align*}
&\displaystyle{\frac{d\Theta(y)}{dy}}=\frac{d\Theta(y)}{dz}\frac{dz}{dy}=\frac{1}{1-(1+\varsigma){{\delta}}}\frac{d\psi(z)}{dz}.
\end{align*}
Since~$\displaystyle{\int_{0}^{\infty}|\psi'|^{2}dz\leq{\delta}c_{1}}$,~we obtain
\begin{align*}
&\quad\ \displaystyle{\int_{0}^{\infty}|\Theta'(y)|^{2}dy}&\\
&=(1-(1+\varsigma){\delta})\int_{0}^{\infty}|\psi'(z)|^{2}\left(\frac{1}{1-(1+\varsigma)\delta}\right)^{2}dz\\
&=\frac{1}{1-(1+\varsigma){\delta}}\int_{0}^{\infty}|\psi'(z)|^{2}dz\\
&\leq\frac{c_{1}{\delta}}{1-(1+\varsigma){\delta}}.
\end{align*}
Subsequently,~$\Theta\in\Lambda_{\frac{c_{1}{\delta}}{1-(1+\varsigma){\delta}}}$.~Therefore,~
\begin{align*}
&\quad\ \displaystyle{\int_{\frac{c_{1}^{2}(1+\alpha)a}{16}}^{\infty}e^{\phi(s)^{\frac{2+\alpha}{1+\alpha}}-s}ds}\\
&\leq\displaystyle{e^{c_{\varsigma,\alpha}\frac{c_{1}(1+\alpha)}{4}\phi^{\frac{2+\alpha}{1+\alpha}}\left(\frac{c_{1}a}{4}\right)-\frac{c_{1}^{2}(1+\alpha)a}{16}}\int_{0}^{\infty}e^{\psi(z){\phi^{\frac{1}{2}}\left(\frac{c_{1}a}{4}\right)}-[1-(1+\varsigma){\delta}]z}dz}\\
&\ \displaystyle{=\frac{1}{(1-(1+\varsigma){\delta})}e^{c_{\varsigma,\alpha}\frac{c_{1}(1+\alpha)}{4}\phi^{\frac{2+\alpha}{1+\alpha}}\left(\frac{c_{1}a}{4}\right)-\frac{c^{2}_{1}(1+\alpha)a}{16}}\int_{0}^{\infty}e^{\Theta(y)\phi^{\frac{1}{2}}\left(\frac{c_{1}a}{4}\right)-y}dy}\\
&\leq\displaystyle{\frac{1}{1-(1+\varsigma){\delta}}e^{c_{\varsigma,\alpha}\frac{c_{1}(1+\alpha)}{4}\phi^{\frac{2+\alpha}{1+\alpha}}\left(\frac{c_{1}a}{4}\right)-\frac{c_{1}^{2}(1+\alpha)a}{16}}e^{\frac{c_{1}{\delta}}{1-(1+\varsigma){\delta}}\frac{\phi\left(\frac{c_{1}a}{4}\right)}{4}+1}}.
\end{align*}
In the last step,~we use Lemma~\ref{lemma 4.1}~where~$\delta_{0}=\frac{c_{1}\delta}{1-(1+\varsigma)\delta}$,~$c={\phi^{\frac{1}{2}}(\frac{c_{1}a}{4})}$.~
\end{proof}
\end{lemma}
We choose~$\tilde{f}_{m}\in{D_{0,\alpha,rad}^{1,2+\alpha}(B^{+})}$~and~$|\nabla\tilde{f}_{m}|^{2+\alpha}t^{\alpha}$~$\rightarrow\delta_{0}$~weakly as~$m\rightarrow+\infty$.~We define~$f_{m}$~by~$|x|=r=e^{-\frac{s}{2+\beta}}$,~$f(s)=(2+\beta){(\frac{c_{\alpha}}{2+\beta})}^{\frac{1}{2+\alpha}}\tilde{f}(x)$.~Then~the following Lemma holds.
\begin{lemma}\label{lemma 4.4}
For any~$\varsigma>0$,~$\alpha>-1$,~$c_{\varsigma,\alpha}{\left(\displaystyle{\int_{0}^{\infty}{|f'_{m}|}^{2+\alpha}ds}\right)}^{\frac{1}{2+\alpha}}\leq1$~and~$J(\tilde{f}_{m})\not\rightarrow J(0)$,~there exits a sequence~$\{a_{m}\}$~satisfying that
\begin{equation}\label{4.1}
c_{\varsigma,\alpha}f_{m}^{\frac{2+\alpha}{1+\alpha}}(a_{m})-a_{m}=-2\ln(a_{m}),
\end{equation}
where the first point in~$[1,~\infty)$~and
~$a_{m}\rightarrow\infty$~as~$m\rightarrow\infty$~.
\end{lemma}
\begin{proof}
We proceed our proof into three steps.

Step~1:~$a_{m}$~does not exist on~$[0,1)$.~Using the~H$\rm{\ddot{o}}$lder~inequality,~we obtain
\begin{align}\label{4.2}
&\quad\ c_{\varsigma,\alpha}f_{m}(s)-f_{m}(0)&\notag\\
&=c_{\varsigma,\alpha}\displaystyle{\int_{0}^{s}f'_{m}(\tau)\cdot1d\tau}\notag\\
&\leq c_{\varsigma,\alpha}\displaystyle{\left(\int_{0}^{s}{|f'_{m}|}^{2+\alpha}d\tau\right)}^{\frac{1}{2+\alpha}}\left({\int_{0}^{s}1d\tau}\right)^{\frac{1+\alpha}{2+\alpha}}\notag\\
&\leq c_{\varsigma,\alpha}{s}^{\frac{1+\alpha}{2+\alpha}}\displaystyle{\left(\int_{0}^{\infty}{|f'_{m}|}^{2+\alpha}d\tau\right)}^{\frac{1}{2+\alpha}}.
\end{align}
By~$c_{\varsigma,\alpha}{\left(\displaystyle{\int_{0}^{\infty}{|f'_{m}|}^{2+\alpha}ds}\right)}^{\frac{1}{2+\alpha}}\leq1$~and~\eqref{4.2},~for~$s\geq0$,~we obtain
~$c_{\varsigma,\alpha}f_{m}(s)\leq s^{\frac{1+\alpha}{2+\alpha}}$.~Note that~$\frac{1+\alpha}{2+\alpha}>1$~and~$c_{\varsigma,\alpha}>1$.~
We obtain
\begin{align}\label{4.3}
c_{\varsigma,\alpha}f_{m}^{\frac{2+\alpha}{1+\alpha}}(s)-s<0.~
\end{align}
In view of~\eqref{4.3},~we have~$c_{\varsigma,\alpha}f_{m}^{\frac{2+\alpha}{1+\alpha}}(s)-s<-2\ln(s)$~for~$s\in[0,1)$.~Then there is no~$a_{m}$~in~$[0,1)$~such that
\begin{align*}
c_{\varsigma,\alpha}f_{m}^{\frac{2+\alpha}{1+\alpha}}(a_{m})-a_{m}=-2\ln(a_{m}).
\end{align*}
Step~2:~$a_{m}$~exists in~$[1,+\infty)$.~Suppose that there is no~$a_{m}$~in~$[1,+\infty)$.~Next,~in the case
\begin{align}\label{4.4}
c_{\varsigma,\alpha}f_{m}^{\frac{2+\alpha}{1+\alpha}}(s)-s>-2\ln(s)~\text{on}~[1,\infty).~
\end{align}
Choosing~$s=1$~,~$-2\ln(s)=0$,~we conclude that it contradicts~\eqref{4.4}~by~\eqref{4.3}.~Therefore,~we have~$c_{\varsigma,\alpha}f_{m}^{\frac{2+\alpha}{1+\alpha}}(s)-s\leq-2\ln(s)$.~
By assumptions,
\begin{align*}
c_{\varsigma,\alpha}f_{m}^{\frac{2+\alpha}{1+\alpha}}(s)-s<-2\ln(s)~\text{on}~[1,\infty).~
\end{align*}
In view of~$\varsigma>0$,~$\alpha>-1$,~we have~$c_{\varsigma,\alpha}>1$.~Therefore,~
\begin{align*}
e^{f_{m}^{\frac{2+\alpha}{1+\alpha}}(s)-s}<\frac{1}{s^{2}}.
\end{align*}
We select the dominating function
\begin{align*}
g(s)=\begin{cases}
1&\text{in(0,1)},\\\frac{1}{s^{2}}&\text{in}[1,+\infty).\\
\end{cases}
\end{align*}
It is easy to verify that~$g(s)$~is a summable function on~$[0,\infty)$.~We conclude that~$J(\tilde{f}_{m})\rightarrow J(0)$~by the dominated convergence theorem.~Therefore,~it contradicts the assumptions of Lemma~\ref{lemma 4.4}.~

Step~3:~$a_{m}\rightarrow+\infty$~as~$m\rightarrow+\infty$.~By letting~$G$~be an arbitrarily large number~and~$m_{0}$~be a positive real number,~it suffices to prove that for any~$m\geq m_{0}$,~one has~$a_{m}\geq G$.~First,~we choose a small~$\eta$~such that
\begin{align}\label{4.5}
{c_{\varsigma,\alpha}}^{\frac{2+\alpha}{1+\alpha}}\eta^{\frac{1}{1+\alpha}}s<s-2\ln(s),~\text{for any}~s\in[0,G).
\end{align}
Therefore,~from~\eqref{4.2},~
\begin{align*}
\displaystyle{c_{\varsigma,\alpha}f_{m}(s)\leq c_{\varsigma,\alpha}{\left(\int_{0}^{\infty}{|f'_{m}|}^{2+\alpha}ds\right)}^{\frac{1}{2+\alpha}}s^{\frac{1+\alpha}{2+\alpha}}}.
\end{align*}
By Definition~\ref{definition 3.2},~we have
\begin{align*}
\displaystyle{c_{\varsigma,\alpha}^{\frac{2+\alpha}{1+\alpha}}{f_{m}^{\frac{2+\alpha}{1+\alpha}}(s)}\leq {c_{\varsigma,\alpha}}^{\frac{2+\alpha}{1+\alpha}}{\eta}^{\frac{1}{1+\alpha}}s}.
\end{align*}
Since~$c_{\varsigma,\alpha}>1$~and~\eqref{4.5},~we obtain
\begin{align*}
\displaystyle{c_{\varsigma,\alpha}{f_{m}^{\frac{2+\alpha}{1+\alpha}}(s)}}\leq {c_{\varsigma,\alpha}}^{\frac{2+\alpha}{1+\alpha}}{\eta}^{\frac{1}{1+\alpha}}s<s-2\ln s.
\end{align*}
Therefore,~\eqref{4.1}~is not true,~which imply that for~$m\geq m_{0}$,~$a_{m}\geq G$.
\end{proof}

\begin{lemma}\label{lemma 4.5}
Let~$\alpha>0$,~$\beta>-1$,~$\frac{1+\varsigma}{c_{\varsigma,\alpha}}<1$~and~$c_{0} $~be constants.~Then for any~${\|f_{m}\|}_{D_{0,\alpha}^{1,2+\alpha}(B^{+})}\leq \frac{1}{c_{\varsigma,\alpha}}$,~we have
\begin{equation*}
J^{\delta}_{\alpha,\beta}(0)\leq\frac{c_{\beta}}{{(2+\beta)}^{2}c_{0}m_{\beta}(B^{+})}(1+e).
\end{equation*}
\end{lemma}
\begin{proof}
Because~$f_{m}(s)$~is nonnegative,~we obtain
\begin{align*}
\displaystyle{\int_{0}^{\frac{c_{1}(1+\alpha)a_{m}}{4}}e^{f_{m}^{\frac{2+\alpha}{1+\alpha}}(s)-s}}ds\displaystyle{\geq\int_{0}^{\frac{c_{1}(1+\alpha)a_{m}}{4}}e^{-s}ds}\displaystyle{=1-e^{\frac{c_{1}(1+\alpha)a_{m}}{4}}}.
\end{align*}
Then,~$\displaystyle{\int_{0}^{\frac{c_{1}(1+\alpha)a_{m}}{4}}e^{f_{m}^{\frac{2+\alpha}{1+\alpha}}-s}(s)\geq1}$,~as~$m\rightarrow\infty$.~Following Definitions~\ref{definition 3.2},~\eqref{3.1}~and Lemma~\ref{lemma 3.3},~we notice that~$f_{m}\rightarrow0$~uniformly on compact subsets of~$\mathbb{R^{+}}$.~Therefore,~by definition of limit,~we have~for each~$A,~\varsigma>0$~,$f_{m}^{\frac{2+\alpha}{1+\alpha}}(s)\leq\varsigma$,~where~$s\leq A$.~Since~$f_{m}^{\frac{2+\alpha}{1+\alpha}}(s)\leq s-2\ln(s)$~for all~$s\leq a_{m}$,~we obtain the following for some sufficiently small~$\varsigma$:~
\begin{align*}
&\quad\ \displaystyle{\int_{0}^{\frac{c_{1}(1+\alpha)a_{m}}{4}}e^{f_{m}^{\frac{2+\alpha}{1+\alpha}}(s)-s}ds}&\\
&=\displaystyle{\int_{0}^{A}e^{f_{m}^{\frac{2+\alpha}{1+\alpha}}(s)-s}ds+\int_{A}^{\frac{c_{1}(1+\alpha)a_{m}}{4}}e^{f_{m}^{\frac{2+\alpha}{1+\alpha}}(s)-s}ds}\\
&\displaystyle{\ \leq e^{\varsigma}\left(1-e^{-A}\right)+\int_{A}^{\frac{c_{1}(1+\alpha)a_{m}}{4}}s^{-2}ds}\\
&=e^{\varsigma}\left(1-e^{-A}\right)+\left(\frac{1}{A}-\frac{4}{c_{1}(1+\alpha)a_{m}}\right).
\end{align*}
By choosing a sufficiently large~$A$,~we obtain
\begin{align*}
\mathop{\lim}\limits_{m\rightarrow\infty}e^{\varsigma}\left(1-e^{-A}\right)+\left(\frac{1}{A}-\frac{4}{c_{1}(1+\alpha)a_{m}}\right)=1.
\end{align*}
In conclusion,~
\begin{align}\label{4.6}
\lim\limits_{m\rightarrow\infty}\displaystyle{\int_{0}^{\frac{c_{1}(1+\alpha)a_{m}}{4}}e^{f_{m}(s)^{\frac{2+\alpha}{1+\alpha}}-s}ds=1}.
\end{align}
Therefore,~
\begin{align*}
\displaystyle{\lim\limits_{m\rightarrow\infty}\frac{c_{\beta}}{{(2+\beta)}^{2}c_{0}m_{\beta}(B^{+})}\int_{0}^{a_{m}}e^{f_{m}^{\frac{2+\alpha}{1+\alpha}}(s)-s}dt}\leq\frac{c_{\beta}}{{(2+\beta)}^{2}c_{0}m_{\beta}(B^{+})}.
\end{align*}
By using~\eqref{4.2},~we have
\begin{align*}
\displaystyle{f_{m}(s)\leq{\left(\int_{0}^{s}{f_{m}'(\tau)}^{2+\alpha}d\tau\right)}^{\frac{1}{2+\alpha}}s^{\frac{1+\alpha}{2+\alpha}}}.
\end{align*}
Hence
\begin{align}\label{4.7}
\displaystyle{f_{m}^{\frac{2+\alpha}{1+\alpha}}(s)\leq\left(\int_{0}^{s}{f'_{m}(\tau)}^{2+\alpha}d\tau\right)^{\frac{1}{1+\alpha}}}s.
\end{align}
Let~$\delta_{m}=c_{\varsigma,\alpha}^{1+\alpha}\displaystyle{\int_{a_{m}}^{\infty}|f_{m}'(s)|^{2+\alpha}ds}$.~From the constraint condition and~\eqref{4.7},~we have
\begin{align*}
&\quad\quad\quad\quad\quad\quad\quad\delta_{m}\leq\displaystyle{1-c_{\varsigma,\alpha}^{1+\alpha}\int_{0}^{a_{m}}|f'_{m}|^{2+\alpha}ds}&\\
&\quad\quad\quad\quad\ \quad\quad\quad\quad\displaystyle{\leq1-c_{\varsigma,\alpha}^{1+\alpha}\left(\frac{f_{m}^{\frac{2+\alpha}{1+\alpha}}(a_{m})}{a_{m}}\right)^{1+\alpha}}\\
&\quad\quad\quad\quad\ \quad\quad\quad\quad\displaystyle{=1-\left(\frac{c_{\varsigma,\alpha}f_{m}^{\frac{2+\alpha}{1+\alpha}}(a_{m})}{a_{m}}\right)^{1+\alpha}}
\end{align*}
\begin{align*}
&\ \quad\displaystyle{=1-\left(\frac{a_{m}-2\ln(a_{m})}{a_{m}}\right)^{1+\alpha}}\\
&\ \quad\displaystyle{=1-\left(1-\frac{2\ln(a_{m})}{a_{m}}\right)^{1+\alpha}}.
\end{align*}
According to~Lemma~\ref{lemma 4.3},~
\begin{equation*}
\sup\limits_{\phi\in\tilde{\Lambda}_{\delta}}\displaystyle{\int_{\frac{c_{1}^{2}(1+\alpha)a}{16}}^{\infty}e^{\phi(s)^{\frac{2+\alpha}{1+\alpha}}-s}ds}\leq \frac{1}{(1-(1+\varsigma){\delta})}e^{c_{\varsigma,\alpha}\frac{c_{1}(1+\alpha)}{4}\phi^{\frac{2+\alpha}{1+\alpha}}(\frac{c_{1}a}{4})-\frac{c_{1}^{2}(1+\alpha)a}{16}+\frac{c_{1}{\delta}}{1-(1+\varsigma){\delta}}\frac{\phi(\frac{c_{1}a}{4})}{4}+1}.
\end{equation*}
We select ~$\phi=f_{m},~{\delta}=\delta_{m},~\frac{c_{1}a}{4}=a_{m}$,~$k_{m}={c_{\varsigma,\alpha}\frac{c_{1}(1+\alpha)}{4}{f_{m}}^{\frac{2+\alpha}{1+\alpha}}(a_{m})-\frac{1+\alpha}{4}a_{m}c_{1}}+\frac{c_{1}\delta_{m}}{1-(1+\varsigma){\delta_{m}}}\frac{{f_{m}}(a_{m})}{4}$.~Then
\begin{align*}\quad\quad
&\quad\ \displaystyle{\int_{\frac{c_{1}(1+\alpha)a_{m}}{4}}^{\infty}e^{{f_{m}}^{\frac{2+\alpha}{1+\alpha}}-s}ds}&\\
&=\displaystyle{\int_{\frac{c_{1}(1+\alpha)a_{m}}{4}}^{\infty}e^{{f_{m}}^{\frac{2+\alpha}{1+\alpha}}-s}ds}\\
&\leq\frac{1}{(1-(1+\varsigma){\delta_{m}})}e^{c_{\varsigma,\alpha}\frac{c_{1}(1+\alpha)}{4}{f_{m}}^{\frac{2+\alpha}{1+\alpha}}(a_{m})-\frac{1+\alpha}{4}a_{m}c_{1}}\\
&^{+\frac{c_{1}\delta_{m}}{1-(1+\varsigma){\delta_{m}}}\frac{{f_{m}}(a_{m})}{4}+1}\\
&=\frac{1}{(1-(1+\varsigma)\delta_{m})}e^{k_{m}+1}.
\end{align*}
We only need to calculate~$\lim\limits_{m\rightarrow\infty}k_{m}$.~\\

Since~$\delta_{m}\leq1-{\left(1-\frac{2\ln(a_{m})}{a_{m}}\right)}^{1+\alpha}$~as~$a_{m}\rightarrow\infty$,~$\lim\limits_{m\rightarrow\infty}\delta_{m}=0$.~Therefore,~
\begin{align}\label{4.8}
\displaystyle{\lim\limits_{m\rightarrow\infty}\int_{\frac{c_{1}(1+\alpha)a_{m}}{4}}^{\infty}e^{f_{m}^{\frac{2+\alpha}{1+\alpha}}(s)-s}ds}\leq\lim\limits_{m\rightarrow\infty}\frac{1}{(1-(1+\varsigma)\delta_{m})}e^{k_{m}+1}=e^{\lim\limits_{m\rightarrow\infty}k_{m}+1}.~
\end{align}
According to~\eqref{4.1},~
\begin{align*}
k_{m}=-\frac{c_{1}(1+\alpha)}{2}\ln(a_{m})+\frac{c_{1}\delta_{m}}{1-(1+\varsigma)\delta_{m}}\frac{f_{m}(a_{m})}{4}.
\end{align*}
Similarly,~from~\eqref{4.1}~and the condition in the lemma that~$c_{\varsigma,\alpha}>1$,~we obtain
\begin{align*}
f_{m}(a_{m})\leq{(a_{m}-2\ln a_{m})}^{\frac{1+\alpha}{2+\alpha}}\leq {a_{m}}^{\frac{1+\alpha}{2+\alpha}}\leq a_{m}.
\end{align*}
Therefore,~as~$a_{m}\rightarrow\infty$,~
\begin{align*}
k_{m}\leq -\frac{c_{1}(1+\alpha)}{2}\ln (a_{m})+\frac{c_{1}\delta_{m}}{1-(1+\varsigma)\delta_{m}}\frac{{a_{m}}}{4}.
\end{align*}
Let~$r_{m}=1-(1+\varsigma)\delta_{m}$,~we have
\begin{align*}
k_{m}\leq\frac{1}{r_{m}}\left(-\frac{c_{1}(1+\alpha)}{2}\ln(a_{m})\right)r_{m}+\frac{c_{1}\delta_{m}a_{m}}{4r_{m}}.
\end{align*}
By denoting variable~$a_{m}$~as~$x$,~the numerator can be easily calculated as follows:
\begin{align*}
&\quad\quad\quad\quad\ -\frac{c_{1}(1+\alpha)}{2}\ln(x)r_{m}+\frac{c_{1}\delta_{m}}{4}x\\
&\quad\quad\quad\leq\left(-\frac{c_{1}(1+\alpha)}{2}\ln(x)\right)(1-(1+\varsigma)\delta_{m})+\frac{1}{4}c_{1}x\left(1-{\left(1-\frac{2\ln x}{x}\right)}^{1+\alpha}\right)\\
&\quad\quad\quad=-\frac{c_{1}(1+\alpha)}{2}\ln x+\frac{c_{1}(1+\alpha)}{2}\ln x\cdot(1+\varsigma)\delta_{m}+\frac{c_{1}}{4}\left(1-{\left(1-\frac{2\ln x}{x}\right)}^{1+\alpha}\right)x\\
&\quad\quad\quad=-\frac{c_{1}(1+\alpha)}{2}\ln x+\frac{c_{1}}{4}\left(1-{\left(1-\frac{2\ln x}{x}\right)}^{1+\alpha}\right)x+\frac{c_{1}(1+\alpha)(1+\varsigma)}{2}\ln x\left(1-{\left(1-\frac{2\ln x}{x}\right)}^{1+\alpha}\right)\\
&\quad\quad\quad=-\frac{c_{1}(1+\alpha)}{2}\ln x+\frac{c_{1}}{4}\left(1-{\left(1-\frac{2\ln x}{x}\right)}^{1+\alpha}\right)x+\frac{\frac{c_{1}(1+\alpha)}{2}\left(1-{\left(1-\frac{2\ln x}{x}\right)}^{\frac{1}{1+\alpha}}\right)}{\left(\ln x\cdot(1+\varsigma)\right)^{-1}}.
\end{align*}
Since
\begin{align*}
{({(\ln x\cdot(1+\varsigma))^{-1}})}'=\frac{-(1+\varsigma)\frac{1}{x}}{{(\ln x\cdot(1+\varsigma))}^{2}}
\end{align*}
and
\begin{align*}
{\left(\frac{c_{1}(1+\alpha)}{2}\left(1-{\left(1-\frac{2\ln x}{x}\right)}^{\frac{1}{1+\alpha}}\right)\right)}=\frac{c_{1}}{2}{\left(1-\frac{2\ln x}{x}\right)}^{\frac{1}{1+\alpha}-1}\frac{2-2\ln x}{x^{2}},
\end{align*}
then~for~$\frac{\frac{c_{1}(1+\alpha)}{2}\left(1-{\left(1-\frac{2\ln x}{x}\right)}^{\frac{1}{1+\alpha}}\right)}{(\ln x\cdot(1+\varsigma))^{-1}}$,~after using~L'Hopital's~rule,~
\begin{align*}
&\quad\ \frac{\frac{c_{1}}{2}{(1-\frac{2\ln x}{x})}^{\frac{1}{1+\alpha}-1}\frac{2-2\ln x}{x^{2}}}{\frac{-(1+\varsigma)\frac{1}{x}}{{(\ln x(1+\varsigma))}^{2}}}\\
&=\frac{\frac{c_{1}}{2}{(1-\frac{2\ln x}{x})}^{\frac{1}{1+\alpha}-1}\frac{1}{x}(2-2\ln x){(\ln x(1+\varsigma))}^{2}}{-(1+\varsigma)}\\
&=-\frac{1}{1+\varsigma}\frac{c_{1}}{2}{\left(1-\frac{2\ln x}{x}\right)}^{\frac{1}{1+\alpha}-1}\frac{1}{x}(2-2\ln x){(\ln x\cdot(1+\varsigma))}^{2}.
\end{align*}
Therefore~$\frac{\frac{c_{1}(1+\alpha)}{2}\left(1-{\left(1-\frac{2\ln x}{x}\right)}^{\frac{1}{1+\alpha}}\right)}{(\ln x\cdot(1+\varsigma))^{-1}}=0$~as~$x\rightarrow\infty$.~By simple calculation,~we have
\begin{align*}
\quad-\frac{c_{1}(1+\alpha)}{2}\ln x+\frac{{c_{1}}}{4}\left(1-\left(1-\frac{2\ln x}{x}\right)^{1+\alpha}\right)x=\frac{-\frac{c_{1}(1+\alpha)}{2}\ln x\cdot x^{-1}+\frac{c_{1}}{4}\left(1-(1-\frac{2\ln x}{x})^{1+\alpha}\right)}{x^{-1}}.
\end{align*}
After using~L'Hopital's~rule,~
\begin{align*}
&\quad\ \frac{-\frac{c_{1}(1+\alpha)}{2}\ln x\cdot x^{-1}+\frac{c_{1}}{4}\left(1-(1-\frac{2\ln x}{x})^{1+\alpha}\right)}{x^{-1}}&\\
&=-\frac{c_{1}(1+\alpha)}{2}(\ln x-1)-\frac{c_{1}(1+\alpha)}{4}{\left(1-\frac{2\ln x}{x}\right)}^{\alpha}(2-2\ln x)\\
&=\frac{c_{1}(1+\alpha)}{2}(1-\ln x)\left(1-{\left(1-\frac{2\ln x}{x}\right)}^{\alpha}\right).
\end{align*}
Let~$\mu=\alpha$,~$y=1-\frac{2\ln x}{x}$.~Then
\begin{align*}
&\quad\ \frac{c_{1}(1+\alpha)}{2}(1-\ln x)\left(1-{\left(1-\frac{2\ln x}{x}\right)}^{\alpha}\right)&\\
&\geq\frac{c_{1}(1+\alpha)}{2}(1-\ln x)\left(1-\left(1-(1+\alpha)\frac{2\ln x}{x}\right)\right)\\
&=c_{1}{(1+\alpha)}^{2}(1-\ln x)\frac{\ln x}{x}.
\end{align*}
To see the inequality,~we must use the following inequality that can be easily proven using the Taylor expansion:
\begin{align*}
{(1-y)}^{\mu}\geq1-(1+\mu)y,
\end{align*}
for any~$\mu>0$,~$y>0$~small enough.

Since~$c_{1}{(1+\alpha)}^{2}(1-\ln x)\frac{\ln x}{x}=0$~as~$x\rightarrow\infty$,~we have~$-\frac{c_{1}}{2}(1+\alpha)(\ln x-1)-\frac{c_{1}(1+\alpha)}{4}{(1-\frac{2\ln x}{x})}^{\alpha}(2-2\ln x)\geq0$~as~$x\rightarrow\infty$.~Therefore,~$-\frac{c_{1}}{2}(1+\alpha)(\ln x-1)-\frac{c_{1}(1+\alpha)}{4}{(1-\frac{2\ln x}{x})}^{\alpha}(2-2\ln x)=0$~as~$x\rightarrow\infty$.~

In conclusion,~we obtain
\begin{align*}
&\quad\ \lim\limits_{m\rightarrow\infty}J(\tilde{f}_{m})\\
&=\mathop{\lim}\limits_{m\rightarrow\infty}\frac{c_{\beta}}{{(2+\beta)}^{2}c_{0}m_{\beta}(B^{+})}\displaystyle{\int_{0}^{\infty}\left(e^{{|f_{m}|}^{\frac{2+\alpha}{1+\alpha}}}-1\right)e^{-s}}ds\\
&
=\mathop{\lim}\limits_{m\rightarrow\infty}\frac{c_{\beta}}{{(2+\beta)}^{2}c_{0}m_{\beta}(B^{+})}\left(\displaystyle{\int_{0}^{\frac{c_{1}(1+\alpha)a_{m}}{4}}e^{{{|f_{m}|}^{\frac{2+\alpha}{1+\alpha}}}-s}ds}+\displaystyle{\int_{\frac{c_{1}(1+\alpha)a_{m}}{4}}^{\infty}e^{{{|f_{m}|}^{\frac{2+\alpha}{1+\alpha}}}-s}ds-1}\right)\\
&\leq\frac{c_{\beta}}{{(2+\beta)}^{2}c_{0}m_{\beta}(B^{+})}\left(1+e^{\mathop{\lim}\limits_{m\rightarrow\infty}k_{m}+1}-1\right)\\
&\leq\frac{c_{\beta}}{{(2+\beta)}^{2}c_{0}m_{\beta}(B^{+})}e
\end{align*}
from ~\eqref{3.4},~\eqref{4.6}~and~\eqref{4.8}.
\end{proof}
\begin{proof}[Proof of Theorem \ref{theorem1.2}]
We want to prove this theorem by a contradiction.~We find a function that is not in a concentrated sequence such that its functional~$J$~is larger than the upper bound.~Then it indicate that the maximum sequence containing this function is not concentrated.~Finally, from Lemma~\ref{lemma 3.3},~we conclude that there exists extremal functions.

Let~$\Lambda_{\delta}=\left\{\phi\in C^{1}(0,\infty)|\phi(0)=0,~\displaystyle{\int_{0}^{+\infty}{|\phi'|}^{2}ds\leq\delta}\right\}$.~Motivated by~\cite{R2016}~,we define a function~$f\in\Lambda_{1}$~as follows:
\begin{align*}
f(s)=\begin{cases}
\text{$\frac{s}{2},$}&\text{$0\leq s\leq 2,$}\\
\text{$(s-1)^{\frac{1}{2}},$}&\text{$2\leq s\leq e^{2}+1,$}\\
\text{$e,$}&\text{$s\geq e^{2}+1$}.
\end{cases}
\end{align*}
It was verified in~\cite{C1986}~that
\begin{align*}
\displaystyle{\int_{0}^{\infty}e^{f^{2}(s)-s}ds>\frac{2.723}{e}+e>e}.
\end{align*}
By selecting~$\phi_{0}(s)=f^{2\frac{1+\alpha}{2+\alpha}}$,~we obtain
\begin{align*}
\displaystyle{\int_{0}^{+\infty}e^{{\phi_{0}}^{\frac{2+\alpha}{1+\alpha}}-s}ds}>\frac{2.723}{e}+e>e.
\end{align*}
~It remains to prove~$\phi_{0}(s)\in\tilde{\Lambda}_{\frac{1}{c_{\varsigma,\alpha}}}$.~Since
\begin{align*}
&\Gamma(\phi_{0})={\left(\displaystyle{\int_{B^{+}}{|\nabla u|}^{2+\alpha}t^{\alpha}dxdt}\right)}^{\frac{1}{2+\alpha}}&\\
&\quad\quad\ ={\left(\displaystyle{\int_{0}^{\infty}{|\phi_{0}'(s)|}^{2+\alpha}ds}\right)}^{\frac{1}{2+\alpha}}\\
&\quad\quad\ ={\left(\displaystyle{\int_{0}^{+\infty}{{\left(f^{\frac{2+2\alpha}{2+\alpha}}\right)}'}^{2+\alpha}ds}\right)}^{\frac{1}{2+\alpha}}\\
&\quad\quad\ ={\left({\left(\frac{2+2\alpha}{2+\alpha}\right)}^{2+\alpha}\left(\displaystyle{\int_{0}^{2}f^{\alpha}{f'}^{2+\alpha}ds+\int_{2}^{e^{2}+1}f^{\alpha}{f'}^{2+\alpha}ds}\right)\right)}^{\frac{1}{2+\alpha}}\\
&\quad\quad\ ={\left({\left(\frac{1}{2}\right)}^{1+\alpha}{\left(\frac{2+2\alpha}{2+\alpha}\right)}^{2+\alpha}\frac{1}{1+\alpha}+{\left(\frac{2+2\alpha}{2+\alpha}\right)}^{2+\alpha}{(\frac{1}{2})}^{1+\alpha}\right)}^{\frac{1}{2+\alpha}}\\
&\quad\quad\ ={\left({\left(\frac{1}{2}\right)}^{1+\alpha}{\left(\frac{2+2\alpha}{2+\alpha}\right)}^{2+\alpha}\frac{2+\alpha}{1+\alpha}\right)}^{\frac{1}{2+\alpha}}\\
&\quad\quad\ =\frac{2+2\alpha}{2+\alpha}{\left(\frac{2+\alpha}{1+\alpha}\right)}^{\frac{1}{2+\alpha}}{\left(\frac{1}{2}\right)}^{\frac{1+\alpha}{2+\alpha}}\\
&\quad\quad\ =2{\left(\frac{2+\alpha}{1+\alpha}\right)}^{-\frac{1+\alpha}{2+\alpha}}{\left(\frac{1}{2}\right)}^{\frac{1+\alpha}{2+\alpha}}\\
&\quad\quad\ =2{\left({\left(\frac{1+\alpha}{2+\alpha}\right)}^{1+\alpha}\right)}^{\frac{1}{2+\alpha}}{\left(\frac{1}{2}\right)}^{\frac{1+\alpha}{2+\alpha}}\\
&\quad\quad\ ={\left(\frac{1+\alpha}{2+\alpha}\right)}^{\frac{1+\alpha}{2+\alpha}}2^{\frac{1}{2+\alpha}}
\end{align*}
and condition in Theorem\ref{theorem1.2},~we have
\begin{align*}
\Gamma(\phi_{0})={\left(\displaystyle{\int_{B^{+}}{|\nabla u|}^{2+\alpha}t^{\alpha}dxdt}\right)}^{\frac{1}{2+\alpha}}={\left(\frac{1+\alpha}{2+\alpha}\right)}^{\frac{1+\alpha}{2+\alpha}}2^{\frac{1}{2+\alpha}}\leq\frac{1}{c_{\varsigma,\alpha}}.
\end{align*}
Consequently,~$\phi_{0}(s)\in\tilde{\Lambda}_{\frac{1}{c_{\varsigma,\alpha}}}$.~
\end{proof}
Finally,~we obtain the existence of a radial solution to the nonlinear Dirichlet$($mean field type$)$problem under the assumptions of Theorem~\ref{theorem1.2}~by using variational method:
\begin{align*}
\begin{cases}
\text{$-div({|\nabla u|}^{\alpha}t^{\alpha}\nabla u)=\frac{u^{\frac{1}{1+\alpha}}\left(e^{a_{\alpha,\beta}{|u|}^{\frac{2+\alpha}{1+\alpha}}}-1\right)t^{\beta}}{\displaystyle{\int_{B^{+}}\left(e^{a_{\alpha,\beta}{|u|}^{\frac{2+\alpha}{1+\alpha}}}-1\right) u^{\frac{1}{1+\alpha}}t^{\beta}dxdt}}$}&\text{$in~B^{+}$},\\
\text{$u=0$}&\text{$on~\partial B^{+}$},\\
\text{$u>0$}&\text{$on~B^{+}$}.
\end{cases}
\end{align*}

We establishe concentration-compactness principle and prove the existence of minimizers of a weighted Moser-Trudinger inequality in the two-dimensional upper half space under dynamic changes.~The primary technical challenge in this study is finding subtle variable substitution in Lemma~\ref{lemma 4.3}~and selecting the appropriate function substitution in Lemma~\ref{lemma 4.5}~to obtain the existence of minimizers.~
\begin{remark}
In this study,~only results on the two-dimensional upper half space are considered.~The existence of extremal functions of Moser-Trudinger type inequalities with monomial weights in $n$-dimensional space still needs further research.
\end{remark}
\subsection*{Data availability}
All data generated or analysed during this study are included in this published article.

\vskip 2mm \noindent Yubo Ni, \ \ \ {\small \tt terence$_{-}$ni@sina.com}\\
{ \em School of Mathematics and Statistic, Shaanxi Normal
University, Xi'an, 710119, China}

\end{document}